\documentclass[11pt,reqno]{amsart}

\usepackage[utf8]{inputenc}
\usepackage{amsmath,amssymb,amsthm}
\usepackage{mathrsfs}
\usepackage{booktabs}
\usepackage{enumitem}
\usepackage{graphicx}
\usepackage{tikz}
\usetikzlibrary{arrows.meta,positioning,calc}
\usepackage{adjustbox}
\usepackage{nicefrac}
\usepackage[margin=1.15in]{geometry}
\usepackage{url}
\usepackage{xurl}
\usepackage{amscd}
\usepackage[sc]{mathpazo}
\usepackage[pagebackref=true]{hyperref}
\hypersetup{hidelinks,
 pdftitle={Integer knot invariants: inequalities, computations, and open problems},
 pdfauthor={Michal Jablonowski},
 pdfsubject={Knot invariants and inequalities between them},
 pdfkeywords={knot invariants, skein tree depth, interval propagation}}

\usetikzlibrary{
	arrows.meta,
	calc,
	positioning,
	fit,
	backgrounds,
	shapes.geometric,
	bending
}

\definecolor{PosterBlue}{HTML}{123C69}
\definecolor{EdgeBlue}{HTML}{274C77}
\definecolor{NodeFill}{HTML}{FFFFFF}
\definecolor{IneqRed}{HTML}{B11226}
\definecolor{AccentGreen}{HTML}{2A7F62}
\definecolor{LightGray}{HTML}{DDF2D1}
\definecolor{Sunflower}{HTML}{EDF7FF}
\definecolor{Cream}{HTML}{FFFAD3}

\tikzset{
	invariant/.style={
		draw=PosterBlue!75!black,
		fill=NodeFill,
		rounded corners=2.6pt,
		minimum height=8.8mm,
		inner xsep=4.2pt,
		inner ysep=2.6pt,
		font=\large\bfseries,
		align=center,
		rotate=-90,
		transform shape
	},
	poly/.style={
		invariant,
		fill=LightGray,
		draw=AccentGreen!70!black
	},
	homol/.style={
		invariant,
		fill=Sunflower,
		draw=Sunflower!80!black
	},
	4D/.style={
		invariant,
		fill=Cream,
		draw=Cream!80!black
	},
	edgefwd/.style={
		draw=EdgeBlue,
		line width=0.9pt,
		-{Latex[length=3mm,width=2mm]}
	},
	edgeback/.style={
		draw=EdgeBlue,
		line width=0.9pt,
		{Latex[length=3mm,width=2mm]}-
	},
	edgebackwhite/.style={
		draw=NodeFill,
		line width=5pt
	},
	redlabel/.style={
		font=\normalsize,
		text=IneqRed,
		fill=white,
		inner sep=2pt,
		rotate=-90,
		transform shape
	},
}

\newcommand{\VLabel}[1]{{\small\ensuremath{#1}}}
\newcommand{\ELabel}[1]{#1}

\theoremstyle{plain}
\newtheorem{theorem}{Theorem}[section]
\newtheorem{proposition}[theorem]{Proposition}
\newtheorem{lemma}[theorem]{Lemma}
\newtheorem{corollary}[theorem]{Corollary}
\newtheorem{definition}[theorem]{Definition}
\theoremstyle{definition}
\newtheorem{convention}[theorem]{Convention}
\newtheorem{remark}[theorem]{Remark}

\newcommand{\kn}[3]{#1\mathrm{#2}_{#3}}

\newcommand{\Z}{\mathbb{Z}}
\DeclareMathOperator{\sgn}{sign}
\DeclareMathOperator{\Tors}{Tors}
\DeclareMathOperator{\Ord}{Ord}
\newcommand{\gds}{g_{ds}}
\newcommand{\gfree}{g_{f}}
\newcommand{\gcan}{g_{c}}
\newcommand{\gr}{g_{r}}
\newcommand{\gtop}{g_4^{\mathrm{top}}}
\newcommand{\cl}{\mathrm{cl}}
\newcommand{\td}{\mathrm{td}}
\newcommand{\tr}{\mathrm{tr}}
\newcommand{\br}{\mathrm{br}}
\newcommand{\ub}{u_b}
\newcommand{\ulb}{u_{lb}}
\newcommand{\us}{u_s}
\newcommand{\uc}{u_c}
\newcommand{\urw}{u_r^{*}}
\newcommand{\spD}{\mathrm{sp}\Delta_t}
\newcommand{\spV}{\mathrm{sp}V_t}
\newcommand{\spPa}{\mathrm{sp}P_a}
\newcommand{\spFa}{\mathrm{sp}F_a}
\newcommand{\degPz}{\deg P_z}

\begin{document}

\title[Integer knot invariants]{Integer knot invariants: inequalities,
computations,\\ and open problems}

\author{Micha\l{} Jab\l{}onowski}
\address{Institute of Mathematics, Faculty of Mathematics, Physics and
Informatics, University of Gda\'nsk, 80-308 Gda\'nsk, Poland}
\email{michal.jablonowski@gmail.com}

\date{September 15, 2026}


\begin{abstract}
We study how inequalities among integer knot invariants combine to yield new exact values and sharper constraints. A directed network of forty-seven established inequalities among thirty-three invariants is combined with parity conditions, alternating evaluations, known ribbon status, and interval propagation for prime knots through thirteen crossings. We prove that the propagation is sound, terminating, and order independent. Under a stated source-bound hypothesis, the resulting database reports, as improvements over the source intervals, exact unknotting numbers for thirty-six thirteen-crossing knots and exact doubly slice genera for seventy-two thirteen-crossing knots. Both lists are given explicitly, and we distinguish rows derived only from genus data from those inheriting an imported unknotting bound. Our main theoretical result is a leading-coefficient obstruction to the depth of ordinary oriented binary skein trees. Its specialisation to the Conway polynomial yields two simpler tests using columns already present: a leading-coefficient test and, as its weakest case, a degree-drop test. Combined with previously computed upper bounds, the obstruction determines eleven skein depths exactly; two of these are not detected by the Conway specialisation. We also formulate nine candidate comparisons not implied by the network, prove family cases and structural restrictions for eight of them, and show that the ninth is equivalent to the smooth Slice--Ribbon Conjecture.
\end{abstract}

\maketitle

\section{Introduction}

Integer-valued knot invariants encode different aspects of the same knot, but
most are difficult to determine individually and many are known only through
upper and lower bounds. Their usefulness therefore increases when inequalities
between them are treated as a system rather than one at a time. Diagrammatic
complexity, Seifert-surface invariants, four-dimensional genera, concordance
invariants, knot homologies, and polynomial degrees all contribute constraints
of different kinds. Building on the earlier inequality network in
\cite{Jab26}, we organise a selected collection of these constraints and study
what can be deduced from their interaction. The displayed arrows in this paper
are established results from the literature; the new mathematical arguments
concern the skein-depth obstruction and the candidate comparisons, while the
new computational component concerns the systematic propagation of interval
data.

The first contribution is a directed, acyclic graph of $47$ established
inequalities among $33$ normalised integer quantities. The normalisation is
chosen so that an arrow has the uniform meaning that its source is at least its
target. The graph mixes crossing-type invariants, classical and canonical
genera, four-dimensional and concordance quantities, Floer- and
Khovanov-theoretic invariants, and polynomial degrees and spreads. Several
arrows are intentionally weaker than sharper published statements because the
graph is designed around a fixed vertex set; the sharper forms are recorded
separately in Subsection~\ref{sec:next}. We also distinguish carefully between
the smooth and locally flat categories, in particular when Alexander-polynomial
bounds are compared with four-dimensional genera.

The second contribution is a reproducible interval-propagation framework for
combining these inequalities with source data. Starting from KnotInfo and the
additional sources cited below, we use the $47$ arrows together with parity
constraints, three non-triviality bounds, the classical alternating-knot
evaluation $\gcan=\gfree=g$, and exact zero values on knots known to be
ribbon. For each knot, lower bounds propagate upward along an inequality and
upper bounds downward. Under the stated hypothesis that every imported interval contains the true
value, we prove that this procedure is sound, that only finitely many strict
endpoint changes can occur, and that every fair schedule reaches the same fixed
point. The order-independence argument is an instance of
standard monotone chaotic iteration, which we cite explicitly rather than claim
as a new general fixed-point theorem.

Applied to prime knots through thirteen crossings, the resulting database
\textsf{NewDB5} contains $12965$ rows. Under the standing source-bound
hypothesis, the construction reports $36$ improved unknotting-number
determinations, all giving the exact value $u=2$, and $72$ improved doubly
slice genus determinations, all giving the exact value $\gds=4$; every row in
both lists is a thirteen-crossing knot. These lists are recorded explicitly in
Section~\ref{sec:db}. The final workbook also
supports direct fixed-point and witness checks, but it does not contain the
pre-propagation snapshot needed to reconstruct every retrospective change
count. We therefore separate reproducible properties of the supplied final
assignment from historical statements about how many bounds were improved in
the original construction. The same distinction is made for the reported
skein-depth change counts.

Our main new theoretical result is a leading-coefficient obstruction for the
depth of an ordinary oriented binary skein tree. The usual degree estimate
$\td(K)\geq\degPz(K)$ follows from the skein expansion. We analyse the equality
case and show that if a depth-minimising tree attains this degree, then the top
coefficient of the HOMFLY--PT polynomial must be a single monomial of a very
restricted form. Specialisation to the Conway polynomial gives two successively weaker but
simpler tests: a leading-coefficient test and a degree-drop test, both using
columns already present in the database. Combined with the previously reported
depth-five upper bounds of \cite{Jab26}, the full obstruction determines the
skein depth exactly for eleven knots, two of which pass the Conway
specialisation $a=1$ and are detected only after retaining the $a$-dependence
of the top HOMFLY--PT coefficient. The historical construction reports $443$ cases in which the polynomial
tests strengthen lower bounds; as above, that count is treated as retrospective
because the earlier source intervals are not part of the supplied final
workbook.

The final part of the paper uses the network as a screening device for possible
new inequalities. We retain nine comparisons that are not consequences of the
transitive closure of the $47$ established arrows, are not contradicted by the
recorded intervals, and have both certified equality and strict witnesses in
the database. None of these candidates is used to construct the database. We
prove family cases and structural restrictions for eight of them, covering
alternating, torus, fibred, homogeneous, simultaneously signature-thin,
quasi-alternating, and positive knots. For the remaining comparison, between
ribbon genus and slicing number, we prove that its universal validity is
equivalent to the smooth Slice--Ribbon Conjecture. Thus, the candidate section separates
empirical compatibility from theorem-level evidence and from statements whose
full resolution would settle a major open problem.

The paper is organised as follows. Section~\ref{sec:network} fixes the
conventions, displays the graph, and records the references and short arguments
justifying its arrows. Section~\ref{sec:main-results} proves the
leading-coefficient obstruction and the eleven exact skein depths.
Section~\ref{sec:db} describes the database, its initialisation and
propagation, the soundness and confluence results, the scope of the
computations, and the exact values singled out by the construction.
Section~\ref{sec:candidates} introduces the nine candidate comparisons and
proves the partial results and counterexample restrictions.
Appendix~\ref{app:defs} collects the definitions of all $33$ vertex quantities,
and Appendix~\ref{app:table} gives the machine-readable form of the $47$
established relations.

\section{The inequality network}\label{sec:network}

We present the graph of the selected $47$ invariant inequalities for any knot
$K \hookrightarrow S^3$. In the diagram, shown in Figure~\ref{fig:graph}, we use
the source/target convention that an arrow $X \to Y$ means $X(K) \ge Y(K)$ for
every non-trivial knot $K$. We stress that the vertices are the normalised
quantities written inside the boxes, and the arrows compare precisely those quantities. Some
of the displayed relations fail for the unknot, which is why non-triviality is
assumed for the graph; relation $38$, for instance, reads $m(K)-1\le\alpha(K)-2$,
while the unknot has $m=2$ and $\alpha=2$. The knot invariants used below are abbreviated as
follows; definitions are collected in Appendix~\ref{app:defs}.

\medskip
\noindent\emph{Crossing-type and diagrammatic invariants.}
\begin{itemize}[leftmargin=2.2em,itemsep=0pt]
\item $c$ --- the crossing number,
\item $c_3$ --- the triple-crossing number,
\item $c_\Delta$ --- the delta-crossing number,
\item $\gfree$ --- the free genus,
\item $\gcan$ --- the canonical genus,
\item $g$ --- the genus,
\item $u$ --- the unknotting number,
\item $\ub$ --- the band-unknotting number,
\item $\ulb$ --- the band-unlinking number,
\item $\cl$ --- the clasp number,
\item $\td$ --- the skein tree depth,
\item $\tr$ --- the trivializing number,
\item $\br$ --- the bridge number,
\item $b$ --- the braid index,
\item $\alpha$ --- the arc index,
\item $m$ --- the mosaic number,
\item $a$ --- the ascending number.
\end{itemize}

\noindent\emph{Four-dimensional and concordance invariants.}
\begin{itemize}[leftmargin=2.2em,itemsep=0pt]
\item $\cl_4$ --- the 4D clasp number,
\item $g_4$ --- the slice genus,
\item $\sigma$ --- the knot signature,
\item $\gds$ --- the doubly slice genus,
\item $\gr$ --- the ribbon genus,
\item $\urw$ --- the weak ribbon unknotting number,
\item $\us$ --- the slicing number,
\item $\uc$ --- the concordance unknotting number.
\end{itemize}

\noindent\emph{Quantities from knot homologies.}
\begin{itemize}[leftmargin=2.2em,itemsep=0pt]
\item $\Ord_v$ --- the torsion order,
\item $\tau$ --- the Ozsv\'ath--Szab\'o $\tau$-invariant,
\item $s$ --- the Rasmussen $s$-invariant.
\end{itemize}

\noindent\emph{Quantities from knot polynomials.}
\begin{itemize}[leftmargin=2.2em,itemsep=0pt]
\item $\spD$ --- the span of the Alexander polynomial $\Delta(t)$,
\item $\spV$ --- the span of the Jones polynomial $V(t)$,
\item $\degPz$ --- the $z$-degree of the HOMFLY-PT polynomial $P_K(a,z)$,
\item $\spPa$ --- the $a$-spread of the HOMFLY-PT polynomial $P_K(a,z)$,
\item $\spFa$ --- the $a$-spread of the Kauffman polynomial $F_K(a,z)$.
\end{itemize}

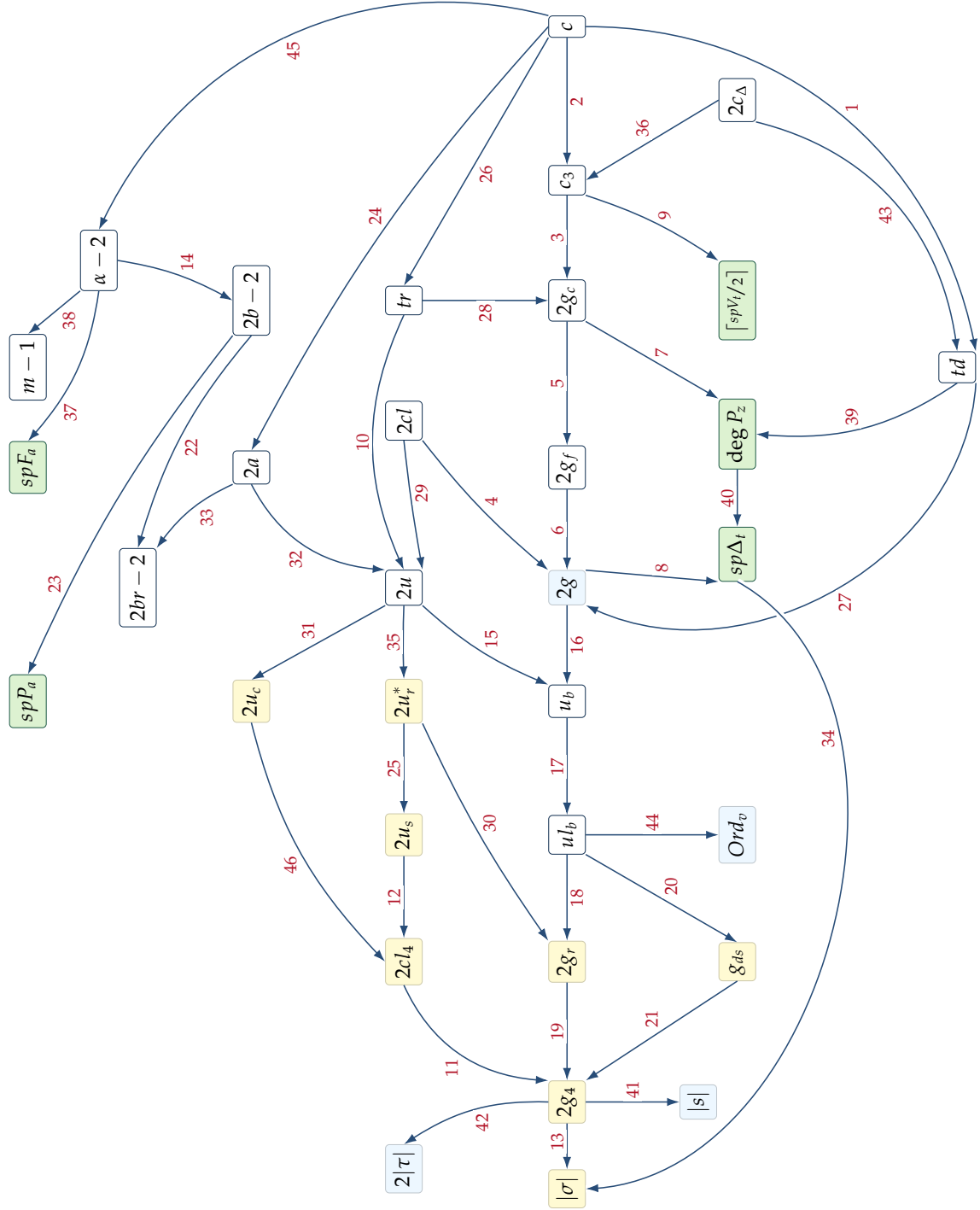
\begin{figure}[p]
	\begin{center}
		
		\begin{adjustbox}{
				max width=\textwidth,
				max height=\textheight,
				center
			}
			
			\rotatebox{90}{
				
				\begin{tikzpicture}[
					x=1.6cm,
					y=1.7cm
					]
					
					\node[invariant] (b)
					at (12,8.4)
					{\VLabel{2b-2}};
					
					\node[invariant] (br)
					at (7.7,10)
					{\VLabel{2\mathrm{br}-2}};
					
					\node[4D] (uc)
					at (6.0,8.4)
					{\VLabel{2u_c}};
					
					\node[invariant] (asc)
					at (9.5,8.4)
					{\VLabel{2a}};

					\node[4D] (cl4)
					at (2.1,6.25)
					{\VLabel{2\mathrm{cl}_4}};
					
					\node[4D] (us)
					at (4,6.25)
					{\VLabel{2u_s}};
					
					\node[4D] (ur)
					at (6.0,6.25)
					{\VLabel{2u_r^\ast}};
					
					\node[invariant] (u)
					at (7.7,6.25)
					{\VLabel{2u}};
					
					\node[invariant] (cl)
					at (10.2,6.25)
					{\VLabel{2\mathrm{cl}}};
					
					\node[invariant] (tr)
					at (12.0,6.25)
					{\VLabel{\mathrm{tr}}};

					\node[4D] (g4)
					at (0.0,3.95)
					{\VLabel{2g_4}};
					
					\node[4D] (gr)
					at (2.1,3.95)
					{\VLabel{2g_r}};
					
					\node[invariant] (ulb)
					at (4,3.95)
					{\VLabel{u_{lb}}};
					
					\node[invariant] (ub)
					at (6.0,3.95)
					{\VLabel{u_b}};
					
					\node[invariant] (g)
					at (7.7,3.95)
					{\VLabel{2g}};
					
					\node[invariant] (gf)
					at (9.5,3.95)
					{\VLabel{2g_f}};
					
					\node[invariant] (gc)
					at (12.0,3.95)
					{\VLabel{2g_c}};
					
					\node[invariant] (c3)
					at (13.8,3.95)
					{\VLabel{c_3}};
					
					\node[invariant] (c)
					at (16.1,3.95)
					{\VLabel{c}};

					\node[4D] (sig)
					at (-1.3,3.95)
					{\VLabel{|\sigma|}};
					
					\node[4D] (gds)
					at (2.1,1.55)
					{\VLabel{g_{ds}}};
					
					\node[homol] (ordv)
					at (4,1.55)
					{\VLabel{\mathrm{Ord}_v}};
					
					\node[invariant] (td)
					at (11.0,-1.55)
					{\VLabel{\mathrm{td}}};
					
					\node[poly] (spD)
					at (8.2,1.55)
					{\VLabel{\mathrm{sp}\Delta_t}};
					
					\node[poly] (degPz)
					at (10.0,1.55)
					{\VLabel{\deg P_z}};
					
					\node[poly] (spV)
					at (12.0,1.55)
					{\VLabel{\left\lceil
							\nicefrac{\mathrm{sp}V_t}{2}
							\right\rceil}};
					
					\node[invariant] (cD)
					at (15.0,1.55)
					{\VLabel{2c_{\Delta}}};

					\node[homol] (tau)
					at (-1.0,6.25)
					{\VLabel{2|\tau|}};
					
					\node[homol] (s)
					at (-0.0,2.1)
					{\VLabel{|s|}};
					
					\node[poly] (spFa)
					at (9.5,11.55)
					{\VLabel{\mathrm{sp}F_a}};
					
					\node[invariant] (m)
					at (11.0,11.55)
					{\VLabel{m-1}};
					
					\node[invariant] (alpha)
					at (12.6,10.55)
					{\VLabel{\alpha-2}};
					
					\node[poly] (spPa)
					at (6.0,11.55)
					{\VLabel{\mathrm{sp}P_a}};

					\draw[edgeback]
					(b.west)
					to[bend right=10]
					node[redlabel,pos=.60,auto,swap]
					{\ELabel{14}}
					(alpha.east);
					
					\draw[edgeback]
					(br.north)
					to[bend left=12]
					node[redlabel,pos=.50,auto,swap]
					{\ELabel{22}}
					(b.south);
					
					\draw[edgeback]
					(br.north east)
					to[bend left=15]
					node[redlabel,pos=.60,auto,swap]
					{\ELabel{33}}
					(asc.south west);
					
					\draw[edgeback]
					(uc.north)
					--
					node[redlabel,pos=.50,auto]
					{\ELabel{31}}
					(u.south west);
					
					\draw[edgefwd]
					(asc.south)
					to[bend right=30]
					node[redlabel,pos=.50,auto,swap]
					{\ELabel{32}}
					(u.north west);
					
					\draw[edgeback]
					(asc.north)
					to[bend left=7]
					node[redlabel,pos=.50,auto]
					{\ELabel{24}}
					(c.west);
					
					\draw[edgeback]
					(spPa.north)
					to[bend left=7]
					node[redlabel,pos=.40,auto]
					{\ELabel{23}}
					(b.south west);
					
					\draw[edgefwd,rounded corners=45pt]
					(br.south)
					-- (5.15,9.35)
					-- node[redlabel,pos=.3,auto,swap]
					{\ELabel{47}}
					(5.15,2.15)
					-- (ordv.north);
					
					\draw[edgeback]
					(cl4.north)
					--
					node[redlabel,pos=.50,auto]
					{\ELabel{12}}
					(us.south);
					
					\draw[edgebackwhite]
					(us.north)
					--
					(ur.south);
					\draw[edgeback]
					(us.north)
					--
					node[redlabel,pos=.20,auto]
					{\ELabel{25}}
					(ur.south);
					
					\draw[edgeback]
					(ur.north)
					to[bend left=2]
					node[redlabel,pos=.50,auto]
					{\ELabel{35}}
					(u.south);
					
					\draw[edgeback]
					(u.north east)
					to[bend left=2]
					node[redlabel,pos=.50,auto,swap]
					{\ELabel{29}}
					(cl.south);
					
					\draw[edgeback]
					(u.north)
					to[bend left=35]
					node[redlabel,pos=.50,auto]
					{\ELabel{10}}
					(tr.south);
					
					\draw[edgeback]
					(tr.north)
					--
					node[redlabel,pos=.50,auto,swap]
					{\ELabel{26}}
					(c.south west);
					
					\draw[edgeback]
					(g4.north)
					--
					node[redlabel,pos=.50,auto]
					{\ELabel{19}}
					(gr.south);
					
					\draw[edgeback]
					(g4.north west)
					to[bend left=32]
					node[redlabel,pos=.50,auto]
					{\ELabel{11}}
					(cl4.south);
					
					\draw[edgeback]
					(g4.north east)
					--
					node[redlabel,pos=.50,auto]
					{\ELabel{21}}
					(gds.south);
					
					\draw[edgeback]
					(gr.north)
					--
					node[redlabel,pos=.50,auto,swap]
					{\ELabel{18}}
					(ulb.south);
					
					\draw[edgebackwhite]
					(gr.north west)
					to[bend left=5]
					(ur.south east);
					
					\draw[edgeback]
					(gr.north west)
					to[bend left=5]
					node[redlabel,pos=.50,auto,swap]
					{\ELabel{30}}
					(ur.south east);
					
					\draw[edgebackwhite]
					(ulb.north)
					--
					(ub.south);
					\draw[edgeback]
					(ulb.north)
					--
					node[redlabel,pos=.20,auto]
					{\ELabel{17}}
					(ub.south);
					
					\draw[edgeback]
					(ub.north west)
					to[bend left=10]
					node[redlabel,pos=.50,auto,swap]
					{\ELabel{15}}
					(u.south east);
					
					\draw[edgeback]
					(ub.north)
					--
					node[redlabel,pos=.60,auto,swap]
					{\ELabel{16}}
					(g.south);
					
					\draw[edgeback]
					(g.south east)
					to[bend right=57]
					node[redlabel,pos=.60,auto,swap]
					{\ELabel{27}}
					(td.south east);

					\draw[edgeback]
					(ordv.west)
					--
					node[redlabel,pos=.50,auto,swap]
					{\ELabel{44}}
					(ulb.east);

					\draw[edgeback]
					(g.north)
					--
					node[redlabel,pos=.50,auto]
					{\ELabel{6}}
					(gf.south);

					\draw[edgeback]
					(g.north west)
					to[bend left=6]
					node[redlabel,pos=.50,auto,swap]
					{\ELabel{4}}
					(cl.south east);

					\draw[edgeback]
					(gf.north)
					--
					node[redlabel,pos=.50,auto]
					{\ELabel{5}}
					(gc.south);

					\draw[edgeback]
					(gc.north)
					--
					node[redlabel,pos=.50,auto]
					{\ELabel{3}}
					(c3.south);

					\draw[edgeback]
					(gc.west)
					to[bend left=0]
					node[redlabel,pos=.50,auto]
					{\ELabel{28}}
					(tr.east);

					\draw[edgeback]
					(c3.north)
					to[bend left=0]
					node[redlabel,pos=.50,auto,swap]
					{\ELabel{2}}
					(c.south);

					\draw[edgeback]
					(c3.east)
					--
					node[redlabel,pos=.50,auto]
					{\ELabel{36}}
					(cD.west);

					\draw[edgeback]
					(sig.north)
					--
					node[redlabel,pos=.50,auto]
					{\ELabel{13}}
					(g4.south);

					\draw[edgeback]
					(gds.north)
					--
					node[redlabel,pos=.50,auto,swap]
					{\ELabel{20}}
					(ulb.south east);

					\draw[edgeback]
					(td.north)
					to[bend right=25]
					node[redlabel,pos=.50,auto]
					{\ELabel{43}}
					(cD.south east);

					\draw[edgeback]
					(td.north east)
					to[bend right=40]
					node[redlabel,pos=.50,auto,swap]
					{\ELabel{1}}
					(c.east);

					\draw[edgefwd]
					(td.south)
					to[bend left=20]
					node[redlabel,pos=.50,auto,swap]
					{\ELabel{39}}
					(degPz.east);

					\draw[edgeback]
					(spD.north)
					--
					node[redlabel,pos=.50,auto]
					{\ELabel{40}}
					(degPz.south);

					\draw[edgeback]
					(spD.south west)
					--
					node[redlabel,pos=.50,auto,swap]
					{\ELabel{8}}
					(g.north east);

					\draw[edgeback]
					(degPz.north west)
					--
					node[redlabel,pos=.50,auto,swap]
					{\ELabel{7}}
					(gc.south east);

					\draw[edgeback]
					(spV.north west)
					to[bend right=8]
					node[redlabel,pos=.50,auto,swap]
					{\ELabel{9}}
					(c3.south east);

					\draw[edgeback]
					(tau.north)
					to[bend left=18]
					node[redlabel,pos=.60,auto,swap]
					{\ELabel{42}}
					(g4.west);

					\draw[edgeback]
					(s.west)
					to[bend right=0]
					node[redlabel,pos=.50,auto,swap]
					{\ELabel{41}}
					(g4.east);

					\draw[edgeback]
					(spFa.north)
					to[bend right=17]
					node[redlabel,pos=.80,auto,swap]
					{\ELabel{37}}
					(alpha.south);

					\draw[edgeback]
					(m.north)
					--
					node[redlabel,pos=.60,auto,swap]
					{\ELabel{38}}
					(alpha.south west);

					\draw[edgeback]
					(alpha.north)
					to[bend left=40]
					node[redlabel,pos=.50,auto,swap]
					{\ELabel{45}}
					(c.north west);
					
					\draw[edgebackwhite]
					(cl4.west)
					to[bend left=15]
					(uc.south);
					
					\draw[edgeback]
					(cl4.west)
					to[bend left=15]
					node[redlabel,pos=.50,auto]
					{\ELabel{46}}
					(uc.south);

					\draw[edgebackwhite]
					(sig.east)
					to[bend right=75]
					(spD.south);
					
					\draw[edgeback]
					(sig.east)
					to[bend right=75]
					node[redlabel,pos=.6,auto]
					{\ELabel{34}}
					(spD.south);
					
				\end{tikzpicture}
			}
			
		\end{adjustbox}
		
	\end{center}
	
	\caption{Graph of selected $47$ inequalities between $33$ invariants}
	\label{fig:graph}
	
\end{figure}

\begin{convention}\label{conv:cat}
Unless explicitly stated otherwise, knots and surfaces are smooth and oriented.
The invariants $g_4$, $\gr$, $\gds$, and $\cl_4$ are used in the smooth
category. Besides the $33$ vertex quantities we use one auxiliary invariant
that is \emph{not} a vertex of Figure~\ref{fig:graph}: the topological slice
genus $\gtop(K)$, the minimal genus of a locally flat, oriented surface properly
embedded in $B^4$ with boundary $K$. It satisfies $\gtop(K) \le g_4(K)$, and the
inequality can be strict. We normalise the signature so that the right-handed trefoil has
$\sigma=-2$, and use the standard compatible conventions
$\tau(T(2,3))=1$ and $s(T(2,3))=2$. Thus, the alternating-knot formulas used
below read $s=-\sigma$ and $\tau=-\sigma/2$. Simultaneously reversing all
three sign conventions leaves every absolute-value statement in this paper
unchanged. For the inequalities imported from Shibuya~\cite{Shibuya}, we use the
corresponding knot versions of his local crossing-change and ribbon
constructions. In particular, $\urw$ denotes a crossing-change distance to weak
ribbon knots, as in Shibuya's Definition~5; it is \emph{not} the number of
ribbon singularities of a ribbon disk.
\end{convention}

\begin{convention}\label{conv:poly}
For the HOMFLY--PT polynomial we use the normalisation $P_U(a,z)=1$ and
\[
a\,P_{L_+}(a,z) - a^{-1}P_{L_-}(a,z) = z\,P_{L_0}(a,z).
\]
Thus, $P_K(1,z)$ is the Conway polynomial $\nabla_K(z)$ and, up to multiplication
by a unit $\pm t^{r}$,
\[
\Delta_K(t) = \nabla_K\!\left(t^{1/2}-t^{-1/2}\right).
\]
All polynomial spans and degrees below are taken in this convention; spans are
unaffected by multiplication by monomial units. KnotInfo uses
$v^{-1}P_{L_+}-vP_{L_-}=zP_{L_0}$, so its HOMFLY variable must be replaced by
$v=a^{-1}$ when comparing coefficients with our convention
\cite{KnotInfoPolynomials}. Neither the $z$-degree nor the spread in the other
variable changes under this replacement, but coefficient signs must not be
renormalised arbitrarily.
\end{convention}

References for the corresponding relations, numbered by the arrow labels, are
as follows: $1$ in \cite{Cromwell89}, $2,9$ in \cite{Adams}, $3$ in
\cite{Jab20}, $4,29$ in \cite{Shibuya}, $5$ in \cite{KobKob}, $6$ in
\cite{Moriah}, $7$ in \cite{Morton}, $8$ in \cite{Giller}, $10$ in
\cite{Henrich}, $28$ in \cite{Hanaki14}, $11$ in \cite{MurYas}, $13$ in \cite{KaufTaylor}, $14$ in
\cite{Cromwell95}, $15$ in \cite{HNT}, $16,17,18,44,47$ in \cite{JMZ}, $20$ in
\cite{McDonald}, $21$ in \cite{LivMeier}, $22$ in \cite{Schubert} (see also
\cite{Yamada}), $23$ in \cite{FranksWilliams,Morton}, $24,32,33$ in
\cite{Ozawa}, $26$ in \cite{Hanaki14}, $27$ in \cite{SchThom}, $31,46$ in
\cite{OwensStrle}, $36$ in \cite{Jab23}, $37$ in \cite{MortonBeltrami}, $38$ in
\cite{LHLO}, $39,43$ in \cite{Jab26}, $41$ in \cite{Rasmussen}, $42$ in
\cite{OzsSzaGenus}, $45$ in \cite{BaePark}.

The remaining displayed relations, or useful details about them, are justified
as follows.

\begin{itemize}[leftmargin=1.6em]
\item \textbf{12.} If $n=\us(K)$ crossing changes transform $K$ into a slice
knot, trace those crossing changes in $S^3\times I$ and cap the final knot by a
slice disk in $B^4$. This is the crossing-change trace construction of
Shibuya~\cite[Lemma~10]{Shibuya}. The result is a properly immersed disk for $K$
with at most $n$ transverse double points. Hence, $\cl_4(K)\le \us(K)$, i.e.,\ $2\cl_4 \le 2\us$.

\item \textbf{19.} Every ribbon surface in $B^4$ is, in particular, a smooth
properly embedded surface in $B^4$. Therefore, $g_4(K)\le \gr(K)$.

\item \textbf{20.} McDonald's band number, written $b(K)$ in \cite{McDonald}
and not to be confused with the braid index $b(K)$ used throughout the present
paper, is the minimum number of oriented saddle moves needed to pass from $K$
to an unlink. This is exactly the quantity denoted $\ulb(K)$ here. McDonald's
theorem therefore gives $\gds(K)\le \ulb(K)$.

\item \textbf{22.} The bound $\br(K)\le b(K)$ is classical: present an $n$-braid in
a cylinder with all strands monotone in height, and close it by $n$ disjoint
exterior arcs, each with one maximum and one minimum. This gives a
presentation of the closure with $n$ maxima and hence an $n$-bridge
presentation. We attribute this bound to Schubert \cite{Schubert}; Yamada's identification of the braid index with the minimal
number of Seifert circles \cite{Yamada} is the standard route to the companion
relation $23$.

\item \textbf{25, 30, 35.} With the convention for $\urw$ fixed in
Appendix~\ref{app:defs}, these three comparisons have direct smooth proofs.
Every ribbon knot is slice, so $\us(K)\leq\urw(K)$, which is relation~$25$.
The unknot is ribbon, so $\urw(K)\leq u(K)$, which is relation~$35$. For
relation~$30$, let $n=\urw(K)$ and choose $n$ crossing changes taking $K$ to a
ribbon knot $J$. A crossing change is realisable by two oriented saddle moves
\cite[Section~1.4]{JMZ}. Reversing the sequence if necessary and concatenating
these saddle cobordisms with a ribbon disk for $J$ produces a ribbon surface for
$K$ of genus $n$. Hence $\gr(K)\leq\urw(K)$. Shibuya's Definition~5 is the
historical source of the weak-ribbon crossing-change invariant
\cite{Shibuya}; in particular, $\urw$ here is not the number of ribbon
singularities of a disk.

\item \textbf{34.} In \cite{Feller} the symbol $\deg\Delta_K$ denotes the
\emph{breadth} of the Alexander polynomial, i.e.,\ the difference between its
largest and its smallest exponent; this is exactly the quantity written $\spD(K)$
here. Feller's Theorem~1 states that this breadth is at least \emph{twice} the
topological slice genus, that is,
\[
2\gtop(K)\ \le\ \deg\Delta_K\ =\ \spD(K).
\]
The signature bound of Kauffman and Taylor \cite{KaufTaylor}, which is valid in
the locally flat category, gives $|\sigma(K)| \le 2\gtop(K)$. Combining the two
inequalities yields
$
|\sigma(K)| \le \spD(K).
$
The companion estimate $\spD(K)\le 2g(K)$ is relation $8$; the two together form
the chain $2\gtop\le \spD\le 2g$ recorded in \cite{Feller}.

\item \textbf{40.} With the normalisation fixed in Convention~\ref{conv:poly},
set $a=1$ to obtain the Conway polynomial. The specialisation can only cause
cancellation, so $\delta:=\deg_z\nabla_K \le \degPz(K)$. For a knot,
$\nabla_K$ is a polynomial in $z^2$; writing
$\nabla_K(z)=\sum_{j \text{ even}} c_j z^{j}$ with $c_\delta\neq 0$ and
substituting $z=t^{1/2}-t^{-1/2}$ therefore produces a genuine Laurent
polynomial in $t$, since a monomial $z^{j}$ with $j$ even expands into
integral powers of $t$ ranging from $t^{j/2}$ down to $t^{-j/2}$ and has
$t$-span exactly $j$. Only the top monomial contributes the extreme exponents
$\pm\delta/2$, with the non-zero coefficient $c_\delta$, so they survive;
hence $\spD(K)=\deg_z\nabla_K=\delta$ and therefore
$\spD(K) \le \degPz(K)$.
\end{itemize}

\begin{remark}\label{rem:norm}
Some comparisons in Figure~\ref{fig:graph} are deliberately displayed in the
normalisation determined by the chosen vertices, even when a sharper numerical
estimate is known.
\begin{itemize}[leftmargin=1.6em]
\item Relation $1$ is drawn as $c \ge \td$, whereas $\td(K)\le c(K)-1$ holds for
every non-trivial knot \cite{Cromwell89,Jab26}.
\item The composite of relations $24$ and $32$ gives $u(K)\le c(K)/2$, whereas
the classical bound is $u(K)\le (c(K)-1)/2$, also obtained here by composing
$2u\le\tr$ with Hanaki's $\tr\le c-1$
\cite{Henrich,Hanaki10,Hanaki14}.
\item Relation $47$ is $2\br(K)-2 \ge \Ord_v(K)$. The sharper
Juh\'asz--Miller--Zemke estimate is $\Ord_v(K)\le \br(K)-1$ \cite{JMZ}, or
equivalently $2\br(K)-2 \ge 2\Ord_v(K)$. Thus, relation $47$ is valid but
intentionally weaker; no additional vertex $2\Ord_v$ is introduced.

\item The weakening in the previous item is not forced by the vertex set. The
sharp form is the linear constraint $2\Ord_v(K)\le 2\br(K)-2$ between two
quantities that are \emph{both} already stored, so it can be applied as a side
rule of the propagation exactly as the congruence conditions of
Subsection~\ref{sec:prop} are, without introducing a vertex. It is not part of
the rule set imposed here; see Subsection~\ref{sec:next}.
\end{itemize}
\end{remark}

\begin{remark}\label{rem:top}
Relation $34$ is a topological statement and must not be composed with the
smooth vertices as if it were a smooth one. Feller's theorem bounds $\gtop$, and
$\gtop(K)\le g_4(K)$ with strict inequality for, e.g., the positively clasped
untwisted Whitehead double of the right-handed trefoil: its Alexander polynomial
is trivial, so it is topologically slice by Freedman's theorem \cite{Freedman}
and $\gtop=0$, while $\tau=1$ by Hedden's computation \cite{Hedden}, so
$g_4\ge|\tau|=1$ by \cite{OzsSzaGenus}. Consequently
\[
\spD(K)\ \ge\ 2\gtop(K) \;\;\text{does \emph{not} give}\;\; \spD(K)\ \ge\ 2g_4(K).
\]
Accordingly the only arrow drawn out of $\spD$ towards the four-dimensional part
of the graph is the one landing on $|\sigma|$, and $\gtop$ is not a vertex. The
classical companion bound $\spD(K)\le 2g(K)$ goes in the other direction and is
relation $8$ of Figure~\ref{fig:graph}.
\end{remark}

\section{A polynomial obstruction for skein tree depth}\label{sec:main-results}

The degree bound $\td(L)\geq\deg_zP_L$ is already recorded in
\cite[Proposition~4.3]{Jab26}. The argument below examines the equality case
in the standard skein expansion and obtains an explicit necessary condition
on the leading coefficient. For earlier work on skein-tree complexity,
including minimal trees and exact depths for positive braid closures, see
\cite{SchThom,Thompson89,Kaplan15}.

We use the normalisation of Convention~\ref{conv:poly},
\[
 aP_{L_+}(a,z)-a^{-1}P_{L_-}(a,z)=zP_{L_0}(a,z),
 \;\;\; P_U(a,z)=1.
\]
For a knot or link with $d=\deg_zP_L\geq0$, write
$p_d(a)=[z^d]P_L(a,z)$ and set
\[
 \mathcal M_d=\{(-1)^q a^{2q-d}:q=0,\ldots,d\}.
\]

\begin{theorem}[Leading-coefficient obstruction]\label{prop:td-leading}
If $L$ admits an ordinary oriented binary skein tree, as defined in Appendix~\ref{app:defs}, of depth $d=\deg_zP_L\geq0$, then
$p_d(a)\in\mathcal M_d$. Consequently,
\[
 p_d(a)\notin\mathcal M_d
 \;\;\Longrightarrow\;\;
 \operatorname{td}(L)\geq d+1.
\]

\end{theorem}

\begin{proof}
Solving the skein relation for the parent expresses it as a linear combination of
its two children: at a positive crossing the crossing-change and smoothing
factors are $a^{-2}$ and $a^{-1}z$, at a negative crossing they are
$a^2$ and $-az$. Iterating this at every internal vertex writes $P_L$ as the sum,
over the leaves of the tree, of the product of the factors along the
root-to-leaf path times the polynomial of the leaf; the factor attached to an
edge depends only on the sign of the crossing resolved there, and the
$r$-component unlink at a leaf has polynomial $((a-a^{-1})/z)^{r-1}$. Hence, a
leaf reached by $s$ smoothing edges and $t$ crossing-change edges and
representing the $r$-component unlink contributes an expression of the form
\[
 \varepsilon a^m(a-a^{-1})^{r-1}z^{s-r+1},\;\;\; \varepsilon\in\{\pm1\},
\]
whose $z$-degree is exactly $s-r+1$, since $(a-a^{-1})^{r-1}$ does not involve
$z$. Since $r\geq1$, in a tree of depth $h$ every leaf contribution satisfies
\[
 s-r+1\leq s\leq s+t\leq h.
\]
Consequently $\deg_zP_L\leq h$, recovering the degree bound
\cite[Proposition~4.3]{Jab26}; in particular, $\td(L)\geq d$.
If $h=d$, a contribution to $z^d$ must satisfy $s-r+1=d$, hence $s=d$, $r=1$
and $t=0$.
There is at most one such path: having no crossing-change edge, it chooses the
smoothing child at each internal vertex, and this determines it. Call it the
\emph{all-smoothing path} of the tree. A priori it may terminate at a leaf of
depth $s<d$, or at a leaf representing an $r$-component unlink with $r\ge2$; in
either case its $z$-degree $s-r+1$ is strictly less than $d$, and then no leaf
at all contributes to $z^{d}$, so that $p_d(a)=0$, contradicting
$d=\deg_zP_L$. Hence the all-smoothing path has length exactly $d$, ends at the unknot,
and is the unique leaf contributing to $z^{d}$. If exactly $q$
of its $d$ smoothings are at negative crossings, its contribution is
\[
(a^{-1}z)^{d-q}(-az)^q\cdot1=(-1)^q a^{2q-d}z^d ,
\]
so this single monomial, with $0\leq q\leq d$, is the whole leading
coefficient; that is, $p_d(a)\in\mathcal M_d$. Since no tree of depth $d$ exists when
$p_d(a)\notin\mathcal M_d$, while $\td(L)\geq d$ always, the stated consequence
follows. Reidemeister moves made after a resolution preserve the
HOMFLY--PT polynomial of the current diagram and do not alter the monomial
factor already accumulated along the preceding skein edges, so allowing such
moves at zero cost does not affect the argument.
\end{proof}

\begin{remark}[Parity and mirror invariance]\label{rem:mirror-Md}
For a knot $K$ one has $P_K(a,z)\in\Z[a^{\pm2},z^{2}]$, so
$d=\degPz(K)$ is even and $p_d(a)$ involves only even powers of $a$.
Since $d$ is even, the substitution $a\mapsto a^{-1}$ maps $\mathcal M_d$ onto
itself: writing $q'=d-q$ gives
\[
(-1)^{q}a^{-(2q-d)}=(-1)^{q}a^{2q'-d}=(-1)^{q'}a^{2q'-d},
\]
the last equality because $(-1)^{d}=1$. Consequently the conclusion of
Theorem~\ref{prop:td-leading} is insensitive to the choice of mirror
representative, as is the test derived from it; no chirality convention has to
be fixed before the obstruction is applied. The same observation is what makes
the coefficient conversion of Subsection~\ref{subsec:td-computation}
unambiguous.
\end{remark}

\begin{corollary}[A test using existing NewDB5 columns]\label{cor:td-drop}
For a knot $K$,
\[
 \operatorname{sp}\Delta_t(K)<\deg_zP_K=d
 \;\;\Longrightarrow\;\;
 \operatorname{td}(K)\geq d+1.
\]
\end{corollary}

\begin{proof}
We have $P_K(1,z)=\nabla_K(z)$ and
$\deg_z\nabla_K=\operatorname{sp}\Delta_t(K)$.
A degree drop forces $p_d(1)=0$, whereas every member of
$\mathcal M_d$ takes the value $1$ or $-1$ at $a=1$.
\end{proof}

\begin{corollary}[Conway leading-coefficient test]\label{cor:td-conway}
Let $K$ be a knot, let $d=\degPz(K)$, and write
$\nabla_K(z)=\sum_j c_j z^{j}$ for its Conway polynomial. If
\[
 c_d\notin\{1,-1\}
 \;\;\;\text{then}\;\;\;
 \operatorname{td}(K)\geq d+1.
\]
Equivalently, in terms of the stored columns: either $\spD(K)<d$, in which case
$c_d=0$ and the conclusion is Corollary~\ref{cor:td-drop}; or $\spD(K)=d$, in
which case $|c_d|$ is the absolute value of the extreme coefficient of
$\Delta_K$, and the conclusion holds as soon as that coefficient is not $\pm1$.
\end{corollary}

\begin{proof}
Since $\operatorname{td}(K)\ge d$ always, the only alternative to
$\operatorname{td}(K)\ge d+1$ is $\operatorname{td}(K)=d$, that is, the
existence of an ordinary oriented binary skein tree of depth $d$. In that case
Theorem~\ref{prop:td-leading} gives $p_d(a)=(-1)^{q}a^{2q-d}$ for some
$0\le q\le d$, so $p_d(1)=(-1)^q\in\{1,-1\}$. By
Convention~\ref{conv:poly} we have $P_K(1,z)=\nabla_K(z)$, hence
$p_d(1)=c_d$, and the first assertion follows. For the second, the computation
of item~$40$ of Section~\ref{sec:network} shows that if
$\delta=\deg_z\nabla_K$ then $\spD(K)=\delta$ and the extreme coefficient of
$\Delta_K$ equals $c_{\delta}$ up to the unit $\pm t^{r}$ of
Convention~\ref{conv:poly}. Thus $\spD(K)=d$ gives $\delta=d$ and
$|c_d|$ is the absolute value of the extreme Alexander coefficient, while
$\spD(K)<d$ gives $c_d=0$.
\end{proof}

\begin{remark}[Strict position of the three tests]\label{rem:three-tests}
The three obstructions of this section are strictly nested. Corollary~\ref{cor:td-drop}
is the case $c_d=0$ of Corollary~\ref{cor:td-conway}, and
Corollary~\ref{cor:td-conway} is the specialisation at $a=1$ of
Theorem~\ref{prop:td-leading}. Both inclusions are strict on the present table:
$\kn{8}{}{20}$ and $\kn{8}{}{21}$ have $d=4$, with $c_4=1$ for
$\kn{8}{}{20}$ and $c_4=-1$ for $\kn{8}{}{21}$, so they pass
Corollary~\ref{cor:td-conway} and, among the three
tests stated here, are detected only by the full monomial condition of
Theorem~\ref{prop:td-leading}, while $\kn{7}{}{3}$ has $c_4=2\neq0$ and so is
detected by Corollary~\ref{cor:td-conway} but not by
Corollary~\ref{cor:td-drop}.
Passing any of the three tests does not prove that $\td(K)=d$.
\end{remark}

\begin{theorem}[Eleven exact skein depths]\label{thm:eleven-depths}
For
\[
K\in\{7_3,7_5,8_4,8_6,8_8,8_{11},8_{13},8_{14},8_{15},8_{20},8_{21}\},
\]
one has $\td(K)=5$.
\end{theorem}
\begin{proof}
The upper bounds $\td\leq5$ are supplied by the depth-five trees recorded for
these eleven knots in the skein-depth tables of \cite{Jab26}, obtained there by an explicit
finite search; they are the upper endpoints imported in
Subsection~\ref{subsec:td-computation}. Their recomputed HOMFLY--PT polynomials
have $z$-degree four, and the leading coefficients listed in
Subsection~\ref{subsec:td-computation} lie outside $\mathcal M_4$, so
Theorem~\ref{prop:td-leading} gives the matching lower bounds $\td\geq5$.
Neither bound depends on the choice of root diagram.
\end{proof}

\section{Database}\label{sec:db}

The database \textsf{NewDB5} accompanying this paper has one row for each prime
knot through thirteen crossings. It records the $33$ quantities appearing as
vertices of Figure~\ref{fig:graph}, two knot-identifier columns, and the leading
coefficient $p_d(a)$ of Section~\ref{sec:main-results}. Each invariant cell
carries an integer interval, printed as an exact value when its endpoints agree;
the conventions for printing are fixed in Remark~\ref{rem:sheet}.

\paragraph{Standing hypothesis on the source bounds.}
All statements obtained by interval propagation are conditional on the validity
of the initial intervals: they are assertions about knot invariants only under
the assumption that each initial interval contains the invariant it represents.
In particular, \emph{certified by the current bounds} in
Section~\ref{sec:candidates} means certified from the recorded intervals under
that assumption. Closure of an assignment under the rule set constrains the
recorded intervals against one another; it does not establish an imported
value.

For knots up to $13$ crossings, the initial intervals for the following $19$
quantities are taken from the KnotInfo tables \cite{KnotInfo}:
\[
c,\ 2u,\ 2g,\ 2\br-2,\ 2b-2,\ |\sigma|,\ \alpha-2,\ \spD,\ \lceil \spV/2\rceil,\
\spFa,
\]
\[
2g_4,\ |s|,\ 2|\tau|,\ 2\cl_4,\ 2\cl,\ \gds,\ m-1,\ \spPa,\ \degPz.
\]
The six quantities $c_3$, $2c_\Delta$, $\td$, $\tr$, $2a$, and $2\uc$ are
taken from the references cited in this paper. The remaining eight quantities
$2\gcan$, $2\gfree$, $2\gr$, $2\urw$, $2\us$, $\Ord_v$, $\ub$, and $\ulb$ are
left unknown when no source value is available.

\subsection{Alternating knots and elementary initialisation rules}
\label{sec:altinit}
We first impose one classical exact evaluation. If $K$ is alternating, then
Seifert's algorithm applied to a reduced alternating diagram produces a minimal
genus Seifert surface \cite{Crowell,Murasugi58}, so that $\gcan(K)=g(K)$; since
$g(K)\le \gfree(K)\le \gcan(K)$ holds for every knot, this forces
\begin{equation}\label{eq:alt}
\gcan(K)=\gfree(K)=g(K)\;\;\;\text{for every alternating knot }K.
\end{equation}
We therefore initialise $2\gcan$ and $2\gfree$ with the imported value $2g$ on
the $6729$ alternating knots of the table.

Since all displayed relations are used only for non-trivial knots, and every
knot of the database is non-trivial, we additionally initialise the a priori
lower bounds
\[
\ub \ge 1,\;\;\; \ulb \ge 1,\;\;\; \Ord_v \ge 1.
\]
The first two hold because a non-trivial knot is neither the unknot nor a
completely split unlink, so at least one band surgery is needed in either case;
the third holds because the torsion order vanishes exactly for the unknot
\cite{JMZ}. Together with the evenness of $\ub$ recorded below, the first bound
yields $\ub \ge 2$ for every knot of the database; through the lower-bound
propagation of relation $47$ and the evenness of $2\br-2$, the third bound also
recovers $2\br-2 \ge 2$ on the knots whose bridge number is not imported
exactly. These lower bounds are imposed together with the $47$ arrows.

The parity constraints used in the propagation are the following.
\begin{itemize}[leftmargin=1.6em,itemsep=1pt]
\item every displayed quantity of the form $2X$ is even;
\item $|\sigma|$ and $|s|$ are even for knots;
\item $\spD$ is even, by the symmetry of the Alexander polynomial;
\item $\tr$ is even by Hanaki's results \cite{Hanaki10,Hanaki14};
\item $\ub$ is even, because each oriented band surgery changes the number of
link components by one, whereas both the initial knot and the final unknot have
one component.
\end{itemize}
No parity restriction is imposed here on $\gds$, $\ulb$, or $\Ord_v$.

\paragraph{Attribution of the torsion-order initialisation.}
The non-triviality rule $\Ord_v(K)\ge1$ uses knot Floer homology's detection
of the unknot, not only the definition of torsion order
\cite{OzsSzaDetect,JMZ}. Indeed, torsion order zero makes the torsion
summand of $HFK^-(K)$ vanish. Since its free summand has rank one,
the standard hat/minus exact sequence then gives rank-one
$\widehat{HFK}(K)$, which detects the unknot. Thus, the rule applies to the
non-trivial prime knots in the table.

\subsection{Known ribbon knots}
\label{sec:ribboninit}
For knots, the definition of $\urw$ in Appendix~\ref{app:defs} gives
\[
 K\text{ ribbon}\;\;\Longleftrightarrow\;\;\urw(K)=0.
\]
Ribbonness also gives $\gr(K)=0$, and a ribbon knot is smoothly slice, hence
$\us(K)=\uc(K)=\cl_4(K)=g_4(K)=0$. For $\uc$, the reason is that a slice
knot is smoothly concordant to the unknot, whose unknotting number is zero.
These are definitional consequences for a known ribbon knot, not an
assumption of the Slice--Ribbon Conjecture.

KnotInfo explicitly records that all known slice prime knots through
thirteen crossings are known to be ribbon \cite{KnotInfoRibbon}; its smooth
four-genus documentation identifies the provenance, including the
thirteen-crossing ribbon list supplied by Manolescu and the work of
Gukov--Halverson--Manolescu--Ruehle and Dunfield--Gong
\cite{KnotInfoSlice,GHMR,DunfieldGong}. We use this status statement together
with the exact smooth-slice entries of the table.
There are $511$ such rows, distributed as follows:
\[
\begin{array}{c|rrrrrrr|r}
c&6&8&9&10&11&12&13&\text{Total}\\\midrule
\text{known ribbon rows}&1&3&3&14&30&107&353&511.
\end{array}
\]
The $158$ rows through twelve crossings agree in number with the completed
classification of prime knots through twelve crossings, in which the non-ribbon
cases are excluded by the known slice obstructions and the remaining $158$ knots
were exhibited as ribbon by earlier hand and machine searches; this count and
its provenance are recorded in \cite{GHMR}, whose thirteen- and
fourteen-crossing tables agree with those of \cite{DunfieldGong}. Rows with an unresolved
interval starting at zero are not selected. The ribbon status is imported from
the cited sources; no ribbon disk certificate is reconstructed here.

On these $511$ rows we impose $2\urw=0$, and $2\uc=0$ by their smooth
sliceness. Relations $25$ and $30$ then force $2\us=2\gr=0$, and relations
$12$ and $46$ force $2\cl_4=0$. Ribbonness does not force $\ub$, $\ulb$,
$\Ord_v$, or $\gds$ to vanish, and none of these four columns is constrained
by it: for example, the row of $6_1$ has $2\urw=2\gr=2\us=2\uc=2\cl_4=0$,
while $\ub=2$ and $\gds=1$.

\subsection{Interval propagation}\label{sec:prop}
For a fixed knot $K$, every quantity $X$ is stored as an integer interval
\[
X(K)\in [L_X(K),U_X(K)]\cap \Z,
\]
where an endpoint may initially be infinite. Exact values correspond to
$L_X=U_X$. If the graph contains the established inequality $X\ge Y$, the
elementary propagation step is
\begin{equation}\label{eq:prop}
L_X \longleftarrow \max\{L_X,L_Y\},\;\;\; U_Y\longleftarrow \min\{U_Y,U_X\}.
\end{equation}
Whenever a proved congruence condition $X\equiv r\pmod m$, $m\ge1$, is
available, replace the stored interval by the \emph{interval hull} of
$[L_X,U_X]\cap(r+m\Z)$. Thus, a finite lower endpoint is rounded up to the
least admissible integer and a finite upper endpoint down to the greatest
admissible integer; infinite endpoints remain infinite. An empty intersection is
reported as a contradiction. The admissible set itself need not be an interval:
the hull, together with the persistent congruence rule, is the representation
used here. This includes the stated parity constraints.

A schedule is \emph{fair} if every rule is revisited indefinitely until
termination. Repeated complete sweeps give such a schedule; termination is
certified by a complete sweep in which no endpoint changes.

\begin{proposition}[Soundness and termination of propagation]\label{prop:sound}
Assume that every initial interval contains the true value of the corresponding
invariant, that every true value is finite, and that every inequality and
congruence condition used by the algorithm is valid. Then every interval
produced by \eqref{eq:prop} and by congruence projection still contains the true
value. For a finite database, every run has only finitely many strict endpoint
updates. A fair run reaches a common fixed point after finitely many steps; in
particular, repeated complete sweeps terminate. An arbitrary unfair run need
not reach a common fixed point.
\end{proposition}

\begin{proof}
Suppose $X(K)\ge Y(K)$. Since $Y(K)\ge L_Y$, we have $X(K)\ge L_Y$, so replacing
$L_X$ by $\max\{L_X,L_Y\}$ cannot remove $X(K)$. Likewise
$Y(K)\le X(K)\le U_X$, so replacing $U_Y$ by $\min\{U_Y,U_X\}$ cannot remove
$Y(K)$. Taking the interval hull of the intersection with a congruence class known
to contain the true value preserves that value as well. Thus, every update
preserves all true values, by induction over the update sequence.

For termination, observe that lower endpoints never decrease, upper endpoints
never increase, and every strict update moves an endpoint by at least one.
Consider a single endpoint. If it is infinite it can become finite at most
once, since neither a max with a finite quantity nor an intersection with a
congruence class can return an infinite value once a finite one has been
recorded. Once finite, a lower endpoint is bounded above by the true value of
its invariant: by the soundness statement just proved the interval always
contains that value, and the value is finite by hypothesis. Dually, an upper
endpoint is bounded below by the same value. Being an integer that moves
monotonically by at least one at each strict update, such an endpoint can
therefore change only finitely many times. Since the database is
finite and carries finitely many endpoints, only finitely many strict endpoint
changes occur in total. After the last strict change the assignment is
constant. Under a fair schedule every rule is subsequently applied, so that
constant assignment is fixed by every rule. A complete unchanged sweep is a
finite stopping certificate.

Fairness is used only for the last two assertions. A schedule that stops
applying some rule $r$ after finitely many steps can become constant while $r$
is still enabled: on a row whose $\Ord_v$ lower endpoint has just been raised,
a run that never applies relation $47$ again is constant from that point on,
yet its limit does not satisfy $L_{2\br-2}\ge L_{\Ord_v}$ and is therefore not
the assignment reached by any fair run. Soundness and the finiteness of the
number of strict updates hold for every schedule.
\end{proof}

\begin{proposition}[Confluence of propagation]\label{prop:confluent}
Under the hypotheses of Proposition~\ref{prop:sound}, the fixed point reached by
the propagation does not depend on the order in which the rules are applied: any
two runs that terminate, and in which every rule has been applied at least once
after the last strict update, produce the same interval assignment. In
particular the phrase ``the fixed point of the rule set'', used throughout this
paper, is unambiguous.
\end{proposition}

\begin{proof}
Order interval assignments by componentwise inclusion, writing $x\subseteq y$
when every interval of $x$ is contained in the corresponding interval of $y$.
Each rule $r$, whether an instance of \eqref{eq:prop} or the interval hull of
an intersection with a congruence class, defines a map on assignments that is
\emph{reductive}, $r(x)\subseteq x$, and \emph{monotone}, $x\subseteq y$ implies
$r(x)\subseteq r(y)$. Monotonicity is immediate from the formulas: if
$L_X\le L_X'$ and $L_Y\le L_Y'$ then $\max\{L_X,L_Y\}\le\max\{L_X',L_Y'\}$, and
dually for the upper endpoints. Both intersection with a fixed congruence
class and passage to interval hulls preserve inclusions.

Equivalently, if $x\subseteq y$, then for every invariant $X$ one has
$L_X(x)\ge L_X(y)$ and $U_X(x)\le U_X(y)$. Thus, the $\max$ update of a
lower endpoint preserves the required reversed endpoint inequality, while the
$\min$ update of an upper endpoint preserves the required endpoint inequality;
this is exactly the componentwise inclusion $r(x)\subseteq r(y)$.

Let $x_0=y_0$ be the common initial assignment, let $(x_n)$ and $(y_n)$ be two
terminating runs, and let $x_\infty$ and $y_\infty$ be their limits; by the
terminality assumption each limit is a fixed point of every rule. We show
$y_\infty\subseteq x_n$ for all $n$ by induction. For $n=0$ this is
$y_\infty\subseteq y_0=x_0$. If $y_\infty\subseteq x_n$ and $x_{n+1}=r(x_n)$, then
monotonicity gives $r(y_\infty)\subseteq r(x_n)=x_{n+1}$, and $r(y_\infty)=y_\infty$
because $y_\infty$ is a fixed point of $r$. Hence, $y_\infty\subseteq x_\infty$, and
by symmetry $x_\infty\subseteq y_\infty$.
\end{proof}

\paragraph{Relation to standard fixed-point iteration.}
The preceding order-independence argument is an instance of fair chaotic
iteration for monotone constraint propagators; see
Apt~\cite[Section~2.1, Chaotic Iteration Theorem and Stabilization Corollary]{Apt99}.
In the arXiv version these are Theorem~2.1 and Corollary~1, respectively.
We include the proof to specify our interval conventions, not to claim a new
general confluence theorem.
Termination here uses the bounded integer endpoint changes proved above,
including the convention for upper endpoints that remain infinite.
The hull representation and fair stopping rule have been fixed above. Neither
order independence nor termination can certify a source interval that was
incorrect before the iteration.

\begin{remark}[Finiteness]\label{rem:finite}
The finiteness hypothesis of Proposition~\ref{prop:sound} holds for all $33$
vertex quantities: each is bounded above on every knot, for instance by
$\gr\le g$, $\us\le u$, $\urw\le u$, $\ulb\le \ub\le 2g$, and $\Ord_v\le \br-1$.
Here $\gr\le g$ holds because a Seifert surface of genus $g$, pushed into the
interior of $B^4$ rel boundary, becomes a properly embedded surface whose radial
function has only minima and saddles, hence a ribbon surface.
\end{remark}

\begin{remark}[Spreadsheet convention]\label{rem:sheet}
The propagation is always carried out on the complete interval assignment. The
table is a rendering of the \emph{complete} assignment in
which the interval of a quantity is printed only when a finite upper bound is
known: a blank cell means that no finite upper bound is available, and the
finite lower bound carried by that cell, for instance $2c_\Delta \ge c_3$
through relation $36$ or $2\cl\ge 2g$ through relation $4$, is then not printed.
All fixed-point statements in this paper refer to the completed interval
assignment, in which every blank cell carries its propagated finite lower bound
together with the upper endpoint $+\infty$; on the printed cells the two
descriptions agree.
All candidate-witness counts below use this complete assignment, not zero
lower bounds substituted for blank cells.
\end{remark}

\begin{remark}[A constraint on the imported columns]\label{rem:kmt}
By the theorem of Kauffman, Murasugi and Thistlethwaite
\cite{Kauffman87,Murasugi87,Thistlethwaite}, one has
$\spV(K)\le c(K)$ for every knot, with equality for alternating knots.
The diagram-level converse requires care. Murasugi proves that if $L$ is a
prime link and $D$ is a non-alternating diagram of $L$, then
\[
 \spV(L)<c(D).
\]
See \cite[Theorem~3]{Murasugi87}. The primeness hypothesis is essential: Murasugi
explicitly notes that equality can occur for a non-alternating diagram of a
composite knot. On the other hand, his link-type characterisation
\cite[Theorem~4]{Murasugi87}, specialised to knots, gives
\[
 \spV(K)=c(K)\;\;\Longleftrightarrow\;\; K\text{ is alternating}.
\]
Indeed, in the knot case the connected sums appearing in Murasugi's theorem
are connected sums of alternating knots and hence are alternating.

Every row of the present table is a prime knot. Thus, on a non-alternating
row, a minimal diagram is non-alternating and Theorem~3 gives
$\spV(K)<c(K)$, while on an alternating row Theorem~2 gives
$\spV(K)=c(K)$. Consequently the dichotomy used in the database is valid on
all rows, even though equality for an arbitrary connected diagram does not in
general characterise reduced alternating diagrams. Since $c$ and
$\lceil\spV/2\rceil$ are both imported columns, this constrains them
independently of the propagation and of every relation of
Figure~\ref{fig:graph}: on the $6729$ alternating rows one must have
\[
 \left\lceil\frac{\spV(K)}2\right\rceil
 =\left\lceil\frac{c(K)}2\right\rceil,
\]
and
\[
 \left\lceil\frac{\spV(K)}2\right\rceil
 \le \left\lceil\frac{c(K)}2\right\rceil,
\]
on every row. Equality of the rounded quantities is not confined to the
alternating rows, which is compatible with the strict inequality
$\spV(K)<c(K)$ on prime non-alternating knots: the rounding
$\spV\mapsto\lceil\spV/2\rceil$ can absorb a difference of one when $c$ is
even, and the un-rounded span $\spV$ is not a column of \textsf{NewDB5}.
\end{remark}

\begin{remark}[An identity constraining the Kauffman column]\label{rem:kauffman}
For a non-split alternating link $L$ one has
\[
\spFa(L)\ \ge\ c(L).
\]
This is the alternating case of Thistlethwaite's lower estimate for the
$a$-spread of the Kauffman polynomial of an adequate link
\cite{Thistlethwaite88,Thistlethwaite88b}: a reduced alternating diagram is
adequate, and there the graph-theoretic lower bound supplied by that estimate is
the number of crossings of the diagram. It is stated in this form for
alternating links by Cromwell \cite{Cromwell98}, and the specialisation is
carried out explicitly in \cite{ParkSeo}. Combining it with relations $37$ and $45$ of
Figure~\ref{fig:graph} gives the Morton--Beltrami bound
$\spFa(K)\le \alpha(K)-2$ \cite{MortonBeltrami} and the Bae--Park bound
$\alpha(K)\le c(K)+2$ \cite{BaePark}. Combining these bounds with the
preceding estimate gives, for every alternating knot $K$,
\[
c(K)\ \le\ \spFa(K)\ \le\ \alpha(K)-2\ \le\ c(K),
\]
so that both displayed inequalities are equalities:
\begin{equation}\label{eq:altF}
\spFa(K)=c(K)\;\;\;\text{and}\;\;\; \alpha(K)=c(K)+2
\;\;\;\text{for every alternating knot }K.
\end{equation}
Since $c$, $\spFa$ and $\alpha-2$ are all imported columns, \eqref{eq:altF}
constrains them independently of the propagation and of Remark~\ref{rem:kmt}:
on the $6729$ alternating rows $\spFa=c$ and $\alpha-2=c$, and on every row
$\spFa\le \alpha-2\le c$.

This constraint is not a consequence of closure under the rule set. The vertex
$\spFa$ is a sink of Figure~\ref{fig:graph}: it is the target of relation $37$
and the source of no arrow. An entry of $\spFa$ that is \emph{too small}
therefore does not violate any arrow inequality merely because it is too
small. Its actual value still has to satisfy nonnegativity, congruences, and
any separately imposed exact evaluations. The same caution applies to each of the eight sinks:
$m-1$, $\Ord_v$, $|s|$, $|\sigma|$, $\spFa$, $\spPa$, $\lceil \spV/2\rceil$, and
$2|\tau|$. For these columns, lowering an entry cannot be detected by the arrow
inequalities alone, provided that all separate initial and congruence
constraints remain satisfied. External identities such as \eqref{eq:altF}
or Remark~\ref{rem:kmt} can supply additional checks.
\end{remark}

\subsection{Structure of the database}
The database records the fixed point of the rule set described above: the
initial intervals of Subsections~\ref{sec:altinit} and~\ref{sec:ribboninit},
together with the skein-depth bounds of
Section~\ref{sec:main-results}, closed under the
$47$ arrows and the stated parity constraints. It has $12965$ data rows and
$33$ invariant columns, and all its intervals are nonempty. Being a fixed point
of the propagation rule \eqref{eq:prop}, it satisfies, for every
arrow $X\ge Y$ of Table~\ref{tab:rel} and every row on which both cells are
printed, the endpoint inequalities $L_X\ge L_Y$ and $U_Y\le U_X$; every endpoint
of a quantity subject to a parity rule of Subsection~\ref{sec:altinit} is of the
required parity, and the initial lower bounds $\ub\ge 2$, $\ulb\ge 1$,
$\Ord_v\ge 1$, and $2\br-2\ge 2$ hold in every row. Blank cells, in the sense of
Remark~\ref{rem:sheet}, occur only in the $2c_\Delta$ and $2\cl$ columns. The
graph has exactly three sources, namely $c$, $2\cl$ and $2c_\Delta$; the column
$c$ is imported exactly, whereas for the other two no finite upper bound can
reach the vertex through \eqref{eq:prop} at all, and no blank cell anywhere in
the table admits a finite propagated upper bound, exactly as
Remark~\ref{rem:sheet} requires.

We emphasise what closure under the rule set does and does not give. It says
that the assignment is closed: no interval contradicts a relation of
Table~\ref{tab:rel}, and no further propagation step is available. It does
\emph{not} certify the imported values themselves. A decrease at a sink or an increase at a source cannot, by itself, violate
an arrow inequality. This observation concerns arrows alone: it does not
permit violating an imported exact value, a non-triviality bound, a congruence,
or an independent identity. The two external
identities of Remarks~\ref{rem:kmt} and~\ref{rem:kauffman} exist for this
reason.

The arrow set is acyclic, each of the $47$ arrows is transitively irredundant,
and the nine retained candidates of Section~\ref{sec:candidates} stay
independent of the arrow set even when added jointly; see
Appendix~\ref{app:table}. The table contains $4630$ exact entries in the $\ub$
column and $6863$ exact entries in each of the $2\gcan$ and $2\gfree$ columns.
These are totals in the table, not novelty claims. The eight columns without
imported source values are constrained by the initialisation rules and the
propagation; in particular, the alternating evaluation \eqref{eq:alt}
initialises $2\gcan$ and $2\gfree$ exactly on $6729$ alternating knots, so that
$6863-6729=134$ of the exact entries in each of those two columns are on
non-alternating rows. Their derivation by propagation is part of the reported
construction, rather than a historical fact recoverable from the final values
alone.

The $36$ unknotting numbers and the $72$ doubly slice genera listed below are
exact entries of the supplied snapshot, under the standing source-bound
hypothesis. Their description in the construction as improvements over the
initial source intervals is retrospective: those pre-propagation intervals and
row-level derivation logs are not part of the supplied snapshot. Thus, the
present audit reproduces the final entries, but not their priority or the
historical change count. The same distinction applies to the $443$ and $341$
change counts in Subsection~\ref{subsec:td-computation}.

Both lists below record the rows that the construction reports as improvements;
neither is the list of all rows carrying the stated value. In the supplied
snapshot $5162$ rows have the exact value $2u=4$ and $1981$ rows the exact value
$\gds=4$.

\medskip
\noindent The following $36$ rows, all at thirteen crossings, carry the exact
value $u=2$:
$\kn{13}{n}{128}$, $\kn{13}{n}{152}$, $\kn{13}{n}{365}$, $\kn{13}{n}{391}$,
$\kn{13}{n}{492}$, $\kn{13}{n}{839}$, $\kn{13}{n}{842}$, $\kn{13}{n}{967}$,
$\kn{13}{n}{976}$, $\kn{13}{n}{993}$, $\kn{13}{n}{1082}$, $\kn{13}{n}{1084}$,
$\kn{13}{n}{1146}$, $\kn{13}{n}{1345}$, $\kn{13}{n}{1478}$, $\kn{13}{n}{1524}$,
$\kn{13}{n}{1846}$, $\kn{13}{n}{2453}$, $\kn{13}{n}{2456}$, $\kn{13}{n}{2612}$,
$\kn{13}{n}{2681}$, $\kn{13}{n}{2689}$, $\kn{13}{n}{2729}$, $\kn{13}{n}{3251}$,
$\kn{13}{n}{3263}$, $\kn{13}{n}{3639}$, $\kn{13}{n}{3729}$, $\kn{13}{n}{3811}$,
$\kn{13}{n}{3853}$, $\kn{13}{n}{3881}$, $\kn{13}{n}{4029}$, $\kn{13}{n}{4381}$,
$\kn{13}{n}{4623}$, $\kn{13}{n}{4648}$, $\kn{13}{n}{4672}$, $\kn{13}{n}{4964}$.

\medskip
\noindent The following $72$ rows, all at thirteen crossings, carry the exact
value $\gds=4$:
$\kn{13}{a}{5}$, $\kn{13}{a}{109}$, $\kn{13}{a}{146}$, $\kn{13}{a}{720}$, $\kn{13}{a}{793}$,
$\kn{13}{a}{820}$, $\kn{13}{a}{1020}$, $\kn{13}{a}{1099}$, $\kn{13}{a}{1332}$, $\kn{13}{a}{1419}$,
$\kn{13}{a}{1476}$, $\kn{13}{a}{1544}$, $\kn{13}{a}{1581}$, $\kn{13}{a}{1784}$, $\kn{13}{a}{1901}$,
$\kn{13}{a}{2143}$, $\kn{13}{a}{2151}$, $\kn{13}{a}{2233}$, $\kn{13}{a}{2372}$, $\kn{13}{a}{2377}$,
$\kn{13}{a}{2436}$, $\kn{13}{a}{2507}$, $\kn{13}{a}{2604}$, $\kn{13}{a}{2607}$, $\kn{13}{a}{2875}$,
$\kn{13}{a}{3105}$, $\kn{13}{a}{3122}$, $\kn{13}{a}{3157}$, $\kn{13}{a}{3484}$, $\kn{13}{a}{3513}$,
$\kn{13}{a}{3930}$, $\kn{13}{a}{4005}$, $\kn{13}{a}{4034}$, $\kn{13}{a}{4122}$, $\kn{13}{a}{4225}$,
$\kn{13}{a}{4295}$, $\kn{13}{a}{4422}$, $\kn{13}{a}{4458}$, $\kn{13}{a}{4531}$, $\kn{13}{a}{4727}$,
$\kn{13}{a}{4778}$, $\kn{13}{a}{4807}$, $\kn{13}{a}{4857}$, $\kn{13}{n}{128}$, $\kn{13}{n}{152}$,
$\kn{13}{n}{365}$, $\kn{13}{n}{391}$, $\kn{13}{n}{492}$, $\kn{13}{n}{839}$, $\kn{13}{n}{993}$,
$\kn{13}{n}{1146}$, $\kn{13}{n}{1345}$, $\kn{13}{n}{1478}$, $\kn{13}{n}{1524}$, $\kn{13}{n}{1846}$,
$\kn{13}{n}{2612}$, $\kn{13}{n}{2681}$, $\kn{13}{n}{2689}$, $\kn{13}{n}{2729}$, $\kn{13}{n}{3251}$,
$\kn{13}{n}{3263}$, $\kn{13}{n}{3639}$, $\kn{13}{n}{3729}$, $\kn{13}{n}{3811}$, $\kn{13}{n}{3853}$,
$\kn{13}{n}{3881}$, $\kn{13}{n}{4029}$, $\kn{13}{n}{4381}$, $\kn{13}{n}{4623}$, $\kn{13}{n}{4648}$,
$\kn{13}{n}{4672}$, $\kn{13}{n}{4964}$.

\subsection{Skein-depth computations and their scope}
\label{subsec:td-computation}

\paragraph{Normalisation of the coefficient data.}
The coefficients through ten crossings are those of Stoimenow's HOMFLY table,
in its stated notation conventions \cite{StoimenowTables}; independent
tabulations of the same polynomials, in their own conventions, are
\cite{KnotInfoPolynomials,KnotAtlas}. The Rolfsen numbering
is used through ten crossings, with the Perko duplication removed and the
$10_{83}/10_{86}$ convention agreeing with KnotInfo.
If $h_d(l)$ is the coefficient of $m^d$ in that convention,
the conversion is $l=i/a$, $m=-iz$ with $i^2=-1$, up to mirroring the
representative; these substitutions turn the relation
$lP_{L_+}+l^{-1}P_{L_-}+mP_{L_0}=0$ into the mirror of
Convention~\ref{conv:poly}. For a knot only even powers of $l$ and of $m$ occur,
so $d$ is even, every exponent $e$ with $c_e\neq0$ is even, and $(d+e)/2$ is an
integer. Thus, a coefficient $c_e l^e$ contributes
$c_e(-1)^{(d+e)/2}a^{-e}$ to $p_d(a)$.
The encoding lists only the coefficients of the non-vanishing powers; the
exponents themselves must be restored from the stated conventions before the
test is applied.
Since $d$ is even, $a\mapsto a^{-1}$ preserves $\mathcal M_d$,
so the obstruction is independent of the chosen mirror representative.
In the source encoding, passing the test is equivalent to
$h_d(l)=l^e$ for an even $e$ with $|e|\leq d$.

\paragraph{Scope of the test.}
In the reported construction, the full leading-coefficient condition was
applied to all $249$ prime knots through ten crossings in the normalisation
just described. The supplied coefficients reproduce $151$ failures of the
necessary condition. The reported $97$ strict improvements concern the earlier
lower endpoints, which are not retained in the final workbook. Through
thirteen crossings Corollary~\ref{cor:td-drop} is applied to all $12965$ rows,
using the recorded $z$-degrees and Alexander spreads. The obstruction produces
lower bounds only, so no upper endpoint changes. The reported historical totals are $443$ strict improvements, of which $11$
made an interval exact:
\[
\begin{array}{c|r|r}
\text{Crossings}&\text{Strictly improved lower bounds}&\text{New exact entries}\\
\midrule
3\text{--}6&0&0\\
7&2&2\\
8&9&9\\
9&0&0\\
10&86&0\\
11&10&0\\
12&39&0\\
13&297&0\\
\midrule
\text{Total}&443&11
\end{array}.
\]
Here ``new'' means new relative to the source intervals in the reported
construction, not a priority claim. The final workbook alone does not certify
these change counts; reproducing them requires the original input snapshot. The $346$ degree-drop improvements at $11$--$13$
crossings are a subset of the improvements available at those crossing numbers,
since the full coefficient test is not applied there. Passing the necessary
condition of Theorem~\ref{prop:td-leading} is never a proof that
$\td(L)=d$.

\begin{remark}[Scope of the Conway test]\label{rem:conway-scope}
Corollary~\ref{cor:td-conway} uses only the Conway polynomial, so, unlike the
full leading-coefficient condition, it can be applied unconditionally to every
row of the table: the extreme Alexander coefficient is an imported datum of the
same provenance as the $\spD$ column, and no conversion of a HOMFLY--PT
coefficient string is involved. On the $249$ prime knots through ten crossings,
where the full condition of Theorem~\ref{prop:td-leading} fails on $151$ rows,
the degree-drop condition of Corollary~\ref{cor:td-drop} fires on none, whereas
Corollary~\ref{cor:td-conway} fires on $132$; the remaining $19$ rows carry
$c_d=\pm1$ and therefore pass the specialisation $a=1$ while failing the full
polynomial condition. This numerical split is computed from the stored
low-crossing coefficient and Conway data. The independent recomputation in
Remark~\ref{rem:recompute} covers all knots through nine crossings; the
contribution of the ten-crossing source data to this split remains source-bound
in the same sense as the other imported data.

Among those $19$ rows are $\kn{8}{}{20}$ and $\kn{8}{}{21}$, so two of the
eleven exact depths of Theorem~\ref{thm:eleven-depths} require information
beyond Corollary~\ref{cor:td-conway}. This must not be read as saying that
every one-point specialisation fails. Indeed, for $d=4$ and $a=2$, the
specialised allowed set is
\[
 \{2^{-4},-2^{-2},1,-2^2,2^4\},
\]
whereas $p_4(2)=4$ for $\kn{8}{}{20}$ and $p_4(2)=-2^{-4}$ for
$\kn{8}{}{21}$; neither value lies in that set.

Applying Corollary~\ref{cor:td-conway} at eleven to thirteen crossings therefore
strictly enlarges the set of rows reached by Corollary~\ref{cor:td-drop} there,
at no cost in hypotheses. It is not imposed in the frozen fixed point reported
above; like Remark~\ref{rem:fullcoeff} it is recorded as an available extension
of the input, and it produces lower bounds only, so no upper endpoint and no
printed column other than $\td$ and the blank $2c_\Delta$ cells can change.
\end{remark}

The historical construction reports no changed upper endpoint and no other
changed \emph{printed} cell after closure under the $47$ arrows and parity. It
also reports $341$ strengthened lower endpoints of $2c_\Delta$ through relation
$43$, all in cells that remain blank. These are retrospective comparison
counts, distinct from the directly reproducible fact that the supplied
assignment is already closed and has no empty interval.

\paragraph{Eleven certified exact entries.}
The reported earlier interval was $[4,5]$ and $d=4$ for each of
\[
 7_3,\ 7_5,\ 8_4,\ 8_6,\ 8_8,\ 8_{11},\ 8_{13},\
 8_{14},\ 8_{15},\ 8_{20},\ 8_{21}.
\]
Each has $\operatorname{td}=5$. Coefficients, up to $a\leftrightarrow a^{-1}$,
are
\[
\begin{array}{c|l}
K&p_4(a)\\ \midrule
7_3,\ 7_5&a^{-4}+a^{-6}\\
8_4&-a^2-1\\
8_6,\ 8_{11},\ 8_{14}&-a^{-2}-a^{-4}\\
8_8,\ 8_{13}&a^2+1\\
8_{15}&a^{-4}+2a^{-6}\\
8_{20}&a^2\\
8_{21}&-a^{-4}
\end{array}.
\]
None belongs to
$\mathcal M_4=\{a^{-4},-a^{-2},1,-a^2,a^4\}$.
The first nine knots have more than one leading monomial; the final two
fail the sign requirement even though their leading coefficient is a
monomial. The lower bound is independent of the root diagram. The upper bound of $5$
is witnessed by the explicit finite computer search.

\begin{remark}[An internal check on the coefficient column]\label{rem:pdcheck}
Specialising at $a=1$ compares column \texttt{AJ} with the Alexander column and
uses no further input. By Convention~\ref{conv:poly}, $p_d(1)$ is the
coefficient of $z^{d}$ in $\nabla_K$, so $p_d(1)=0$ exactly when
$\spD(K)<\degPz(K)$, and $p_d(1)$ is the leading coefficient of the Conway
polynomial when $\spD(K)=\degPz(K)$. The supplied snapshot satisfies this
equivalence on all $12965$ rows: the rows with a degree drop are exactly the
rows whose stored coefficient vanishes at $a=1$, and there are $346$ of them,
distributed as $10$, $39$ and $297$ at eleven, twelve and thirteen crossings.
This check does not certify the individual coefficients; it excludes the
mismatches between the two columns that a conversion error would produce.
\end{remark}

\begin{remark}[Independent recomputation on the low-crossing range]
\label{rem:recompute}
The entries of column \texttt{AJ} through nine crossings were recomputed
from scratch, bypassing the tabulated sources entirely. For each of the $84$
prime knots with at most nine crossings a braid word was taken from the
standard tables, the HOMFLY--PT polynomial was evaluated through the Ocneanu
trace on the Hecke algebra $H_n(q)$ in the normalisation of
Convention~\ref{conv:poly}, and the result was expanded in
$z=s-s^{-1}$. The recomputation reproduces, on every one of those rows, the
stored value of $\degPz$, of $\spPa$, of $\spD$, of $|\sigma|$, and the
stored leading coefficient $p_d(a)$ up to the mirror substitution
$a\leftrightarrow a^{-1}$. In particular the eleven leading coefficients
displayed in Subsection~\ref{subsec:td-computation}, and hence the lower
bounds of Theorem~\ref{thm:eleven-depths}, do not depend on the conversion
described in the normalisation paragraph above.

Two labelling conventions must be applied before such a comparison is
meaningful, and both are the ones already fixed there. First, the names
$10_{83}$ and $10_{86}$ are interchanged between the Rolfsen-based tables
and some computational packages. Second, after the removal of the Perko
duplicate the names $10_{162},\dots,10_{165}$ of the present table
correspond to $10_{163},\dots,10_{166}$ of the unreduced list. With these
two conventions applied, the imported $2|\tau|$ column also agrees with a
direct knot Floer computation on all $249$ prime knots through ten
crossings.
\end{remark}

\begin{remark}[The complete coefficient column is a separate input extension]
\label{rem:fullcoeff}
Although the reported construction used the full coefficient test only through
ten crossings, column \texttt{AJ} of the supplied snapshot contains an entry for
every one of its $12965$ rows. Parsing those Laurent polynomials exactly gives
\[
\begin{array}{c|r|r|r}
c&\text{populated entries}&p_d\notin\mathcal M_d&
  \text{additional lower endpoints}\\\midrule
3\text{--}10&249&151&0\\
11&552&362&351\\
12&2176&1486&1447\\
13&9988&7143&6846\\\midrule
\text{Total}&12965&9142&8644
\end{array}.
\]
The last column compares with the \emph{supplied final snapshot}, not with the
historical input intervals. Conditional on the higher-crossing entries being
the correctly normalised leading coefficients of the named knots, replacing
$L_{\td}$ by $\max(L_{\td},d+1)$ at these $8644$ rows and closing under the
same rules strengthens a further $6795$ lower endpoints of $2c_\Delta$.
All of the latter cells remain blank under the printing convention; no upper
endpoint changes, no interval becomes empty, and no further skein depth
becomes exact. No other printed column changes.

This is an exact check of the \emph{stored coefficient strings}, not an
independent computation of all higher-crossing HOMFLY--PT polynomials. Their
source and conversion history must be supplied before promoting this extension
to unconditional knot data. It is not included in the frozen fixed point or
candidate-witness counts reported here, and the supplied workbook is left
unchanged.
\end{remark}

\paragraph{Completeness of the newly exact list for this one-step test.}
A further exact value from the one-step bound and an unchanged input upper
endpoint requires $U_{\operatorname{td}}=d+1$ and
$L_{\operatorname{td}}\leq d$. Besides the eleven knots above, the only
such row is $13\mathrm a_{4878}$, with $d=12$ and interval $[12,13]$.
It is $T(2,13)$ \cite{StoimenowTables}; its top coefficient is $a^{12}$ or $a^{-12}$ and satisfies
the necessary condition. Thus, no additional interval can become exact
by this particular one-step obstruction while all input upper endpoints
are held fixed.
Separately, two already established results each give
$\operatorname{td}(13\mathrm a_{4878})=12$: the upper estimate
$\operatorname{td}(K)\leq c(K)-1$ in~\ref{S1}, and the exact torus-knot
evaluation of \ref{S6}, since $13\mathrm a_{4878}=T(2,13)$. Neither
improvement is one of the $443$ changes and neither is imposed here; the
recorded interval $[12,13]$ is valid but not sharp. The scope of \ref{S6} on
the present table is narrow: the eight torus knots through thirteen crossings
are
\[
3_1,\ 5_1,\ 7_1,\ 9_1,\ \kn{11}{a}{367},\ \kn{13}{a}{4878},\
8_{19},\ 10_{124},
\]
being $T(2,p)$ for $p=3,5,7,9,11,13$ together with $T(3,4)$ and $T(3,5)$, and
the $\td$ entry is already exact and equal to $2g$ on seven of them, the single
exception being $\kn{13}{a}{4878}$, whose recorded interval is $[12,13]$.
Imposing \ref{S6} therefore changes exactly one cell of the supplied snapshot,
namely the one also resolved by~\ref{S1}; no lower endpoint moves and no
interval becomes empty.

\subsection{Sharper forms of displayed relations, and two exact evaluations}
\label{sec:next}

Several of the displayed arrows are weakenings of sharper published statements,
in the sense already illustrated in Remark~\ref{rem:norm}. We collect them here.
None of them is imposed in the propagation used here; each could be imposed
exactly as the congruence conditions of Subsection~\ref{sec:prop} are, without
enlarging the vertex set.

\begin{enumerate}[label=(S\arabic*),leftmargin=2.6em]
\item\label{S1} Relation $1$ is displayed as $c \ge \td$, whereas
$\td(K)\le c(K)-1$ \cite{Cromwell89,Jab26}.
\item\label{S2} Relation $26$ is displayed as $c \ge \tr$, whereas Hanaki proves
$\tr(K)\le c(K)-1$ for every non-trivial knot \cite{Hanaki10,Hanaki14}.
\item\label{S3} Relation $47$ is displayed as $2\br-2 \ge \Ord_v$, whereas
Juh\'asz--Miller--Zemke prove $2\Ord_v(K) \le 2\br(K)-2$ \cite{JMZ}.
\item\label{S4} The composite of relations $24$ and $32$ gives $2u \le c$,
whereas $2u(K)\le\tr(K)\le c(K)-1$
\cite{Henrich,Hanaki10,Hanaki14}.
\item\label{S5} Relations $15$, $17$ and $44$ give $\Ord_v\le 2u$,
whereas Alishahi and Eftekhary prove the sharper bound
$\Ord_v(K)\le u(K)$ \cite[Theorem~1.1]{AlishahiEftekhary}. Equivalently,
$2\Ord_v\le 2u$ is a side constraint between stored quantities.
\item\label{S6} Relations $1$ and $43$ bound $\td$ from above, by $c$ and by
$2c_\Delta$, and relations $27$ and $39$ bound it from below, by $2g$ and
$\degPz$, whereas for torus knots the depth is
known exactly: a strictly positive $n$-braid word of length $\ell$ has closure
of depth $\ell-n+1$ \cite{Kaplan15}, so that
$\td(T(p,q))=(p-1)(q-1)=2g(T(p,q))$ for $2\le p<q$ with $\gcd(p,q)=1$. This is
an exact evaluation on a recognisable family, in the style of \eqref{eq:alt}
and \eqref{eq:trexact}, and can be imposed as an initialisation rule without
enlarging the vertex set.
\end{enumerate}

Two exact evaluations, also due to Hanaki \cite{Hanaki10,Hanaki14}, can be
imposed as initialisation rules in the style of \eqref{eq:alt}:
\begin{equation}\label{eq:trexact}
\tr(K)=2 \iff K \text{ is a twist knot},\;\;\;
\tr(K)=c(K)-1 \iff K \text{ is a } (2,p)\text{-torus knot}.
\end{equation}
The second determines the $\tr$ column exactly on the $(2,p)$-torus knots
of the table, and the first does so on its twist knots.

Here knots are non-trivial, and the torus parameter satisfies $|p|\ge3$.
For the first statement, Hanaki's one-component classification
\cite[Theorem~1.10]{Hanaki10} reduces a projection of trivializing number two
to a twist projection after removing nugatory crossings. Its non-trivial
resolutions are twist knots; conversely, a standard twist projection has
trivializing number two, and a non-trivial knot cannot have trivializing
number zero.
For the second statement, take a diagram $D$ with $c(D)=c(K)$. The inequalities
$\tr(K)\le\tr(D)\le c(D)-1$ force equality throughout whenever
$\tr(K)=c(K)-1$. Hanaki's extremal classification
\cite[Theorem~1.11]{Hanaki10} identifies the underlying projection as the
closed two-braid projection. Opposite successive braid crossings would cancel,
contrary to crossing minimality; hence $K=T(2,p)$. Conversely, for such a knot
$2g(K)=|p|-1=c(K)-1$, and $2g\le\tr\le c-1$ gives equality.
These arguments use imported projection classifications and elementary
consequences, not interval propagation.

\begin{remark}\label{rem:effect}
The chain $c \ge \tr \ge 2\gcan \ge 2\gfree$ formed by relations $26$, $28$ and
$5$ already supplies a finite upper bound in the $2\gcan$ and $2\gfree$ columns
on every row of the table, even though neither column has imported source
values. Imposing~\ref{S2} replaces the bound $2\gfree \le c$ by
$2\gfree \le c-1$ throughout, and hence, by the evenness of $2\gfree$, by
$2\gfree \le 2\lfloor (c-1)/2\rfloor$.

Quantitatively, on the supplied snapshot these side rules act as follows; in
each of the three experiments no lower endpoint moves, no interval becomes
empty, and no column other than the ones named changes. Imposing~\ref{S1}
lowers the upper endpoint of $\td$ on $12413$ rows and makes exactly one further
skein depth exact, namely $\td(13\mathrm a_{4878})=12$. Imposing~\ref{S2} lowers
the upper endpoint of $\tr$ on $2176$ rows and hence, through relations $28$ and
$5$, the upper endpoints of $2\gcan$ and of $2\gfree$ on $825$ rows each, and
makes $78$ further $\tr$ entries exact. Imposing~\ref{S3} and~\ref{S5} together
lowers the upper endpoint of $\Ord_v$ on $12952$ rows and makes $2117$ of its
entries exact. That no interval becomes empty is a consistency check of the
imported columns against the sharper published statements, independent of the
propagation imposed here.
\end{remark}

\section{Candidate inequalities and partial results}\label{sec:candidates}

Because \textsf{NewDB5} contains intervals as well as exact values, one must
distinguish empirical compatibility from verification. For a proposed inequality
$X\le Y$ and a knot with
\[
X\in[L_X,U_X],\;\;\; Y\in[L_Y,U_Y],
\]
the inequality is \emph{contradicted by the current bounds} if $L_X>U_Y$, and it
is \emph{certified by the current bounds} if $U_X\le L_Y$. In the remaining
cases the bounds are merely compatible with the inequality. An equality witness
is certified when both values are exact and equal; a strict witness $X<Y$ is
certified, for example, when $U_X<L_Y$.

The screen used here is the following. A candidate inequality must be between
vertices of Figure~\ref{fig:graph}; it must not follow from the transitive
closure of the $47$ established relations; it must not be contradicted by any
knot in \textsf{NewDB5}; and there should be at least one certified equality
witness and at least one certified strict witness. One should also discard
either orientation of a pair $X,Y$ when examples are known with $X(K)<Y(K)$ and
with $X(J)>Y(J)$. Such examples include $(2g,\degPz)$ from \cite{Morton},
$(\degPz,2\cl)$ and $(2\gcan,2\cl)$ from \cite{BrittJensen}, $(2\gfree,2\cl)$
from \cite{Moriah}.

All nine candidates below pass the recorded-interval and graph-theoretic
screen: each has an exact-equality witness and a strict witness, none is
contradicted by the recorded intervals, and none follows by transitivity from
the $47$ established arrows even after adjoining the other eight candidates.
These tests do not establish universal validity. None of the candidates
(c1)--(c9) is imposed in the database propagation.
\begin{align*}
&\text{(c1)}\ \ 2\br-2 \le \spFa, &&\text{(c6)}\ \ 2\cl_4 \le 2c_\Delta,\\
&\text{(c2)}\ \ 2\br-2 \le \tr, &&\text{(c7)}\ \ \lceil \spV/2\rceil \le \td,\\
&\text{(c3)}\ \ 2\br-2 \le 2c_\Delta, &&\text{(c8)}\ \ 2|\tau| \le \degPz,\\
&\text{(c4)}\ \ 2\gfree \le \td, &&\text{(c9)}\ \ |s| \le \degPz,\\
&\text{(c5)}\ \ 2\gr \le 2\us.
\end{align*}

The following witness counts use the \emph{complete} assignment of
Remark~\ref{rem:sheet}, including propagated lower bounds in cells with infinite
upper endpoint. The strict column counts $U_X<L_Y$; the certified column counts
all $U_X\le L_Y$ and does not require either invariant to be exact. The last
column counts $L_X>U_Y$. All counts use the complete \textsf{NewDB5}
fixed point, without adjoining any of the candidate inequalities.
\begin{center}\small
\begin{tabular}{lrrrr}
\toprule
Comparison & Exact equality & Strict & All certified & Contradicted\\
\midrule
(c1) & $3$ & $12909$ & $12930$ & $0$\\
(c2) & $20$ & $11734$ & $12485$ & $0$\\
(c3) & $1$ & $12334$ & $12538$ & $0$\\
(c4) & $273$ & $153$ & $6865$ & $0$\\
(c5) & $4723$ & $7$ & $5069$ & $0$\\
(c6) & $14$ & $11716$ & $12821$ & $0$\\
(c7) & $38$ & $9023$ & $11206$ & $0$\\
(c8) & $741$ & $12224$ & $12965$ & $0$\\
(c9) & $741$ & $12224$ & $12965$ & $0$\\
\bottomrule
\end{tabular}
\end{center}
Every candidate has both an exact-equality witness and a strict witness;
the trefoil is the single exact-equality witness of (c3). All these numerical
statements retain the source-bound qualification of Section~\ref{sec:db}.
Among the exact equalities recorded for (c5), the $511$ on known ribbon rows
come from $\gr=\us=0$ by definition; they therefore provide no evidence for
the general Slice--Ribbon problem.

\begin{remark}\label{rem:c4witness}
For (c4) the certified witnesses deserve a separate word. By \eqref{eq:alt} the
$2\gfree$ column is exact on the $6729$ alternating rows, and (c4) is already a
theorem there: an alternating diagram is homogeneous in the sense of
Definition~\ref{def:homogeneous}, so Proposition~\ref{prop:c4hom} applies. These rows therefore provide no evidence
beyond that theorem. Among the non-alternating rows, consider the $134$ on
which $2\gfree$ is exact in the supplied snapshot. On that set (c4) is certified with
exact equality on $46$ rows, the first being $\kn{8}{}{19}$, and strictly on
$35$ rows, the first being $\kn{8}{}{20}$. The remaining $53$ rows also certify
the weak inequality: they satisfy $U_{2\gfree}=L_{\td}<U_{\td}$. They certify
neither exact equality nor strictness, but must not be classified as merely
compatible. The $\td$ column is exact on $70$ of these $134$ rows; this
exactness is not a prerequisite for certifying (c4). In total,
$6865-6729=136$ non-alternating rows certify (c4); the other two have non-exact
$2\gfree$ and lie outside this $134$-row subset.

Since fibredness is not a column of \textsf{NewDB5}, splitting these witnesses
further into fibred and non-fibred rows requires an additional source column.
Witnesses on fibred rows are covered by Proposition~\ref{prop:c4}, and those on
homogeneous rows by Proposition~\ref{prop:c4hom}. Only rows simultaneously
non-fibred and non-homogeneous can provide evidence beyond these proved cases.
By Remark~\ref{rem:effect} every row already carries a finite upper bound for
$2\gfree$; the unresolved rows are limited by the relative interval endpoints,
not simply by the absence of exact values.
\end{remark}

\subsection{Bridge and genus criteria}

The following results apply established inequalities and structural theorems
to the candidate comparisons. A sufficient condition obtained by composing
known arrows is distinguished from a new inequality valid for all knots.

\begin{proposition}\label{prop:c2c3}
Candidates \textup{(c2)} and \textup{(c3)} hold for every knot $K$ with
$\br(K)\le 3$.
\end{proposition}

\begin{proof}
The candidates, like the graph, are asserted for non-trivial knots; we note in
passing that for the unknot $\br=1$ and $2\br-2=\tr=2c_\Delta=0$, so that both
assertions hold there as well. Let $K$ be non-trivial.

Hanaki proved that $\tr(K)$ is even and $\tr(K)\ge 2$ \cite{Hanaki14}. If
$\br(K)=2$, this gives
\[
2\br(K)-2 = 2 \le \tr(K).
\]
If $\br(K)=3$ and $\tr(K)=2$, Hanaki's classification of trivializing-number-two projections implies that
$K$ is a twist knot \cite[Theorem~1.10]{Hanaki10}; see also \cite{Hanaki14}. A twist knot has a
standard two-bridge presentation, contradicting $\br(K)=3$. Hence, the even
integer $\tr(K)$ is at least $4=2\br(K)-2$. This proves (c2) for $\br\le 3$.

For (c3), if $\br(K)=2$ then $c_\Delta(K)\ge 1$, since a zero-delta-crossing
knot diagram represents the unknot. Thus, $2\br(K)-2=2\le 2c_\Delta(K)$. If
$\br(K)=3$ and $c_\Delta(K)=1$, replace the single delta-crossing block by its
defining three ordinary crossings. This gives an ordinary diagram of $K$ with
three crossings, hence $c(K)\le 3$. The only non-trivial knot with at most three
crossings is the trefoil, which has bridge number two. This is again a
contradiction. Therefore, $c_\Delta(K)\ge 2$, and $2\br(K)-2=4\le 2c_\Delta(K)$.
\end{proof}

\begin{proposition}\label{prop:c2c3gc}
Candidates \textup{(c2)} and \textup{(c3)} hold for every non-trivial knot $K$
satisfying
\[
\br(K)\le \gcan(K)+1.
\]
\end{proposition}

\begin{proof}
For (c2), relation $28$ gives $\tr(K)\ge 2\gcan(K)$, and therefore
\[
2\br(K)-2\le 2\gcan(K)\le \tr(K).
\]
For (c3), relations $36$ and $3$ give
\[
2c_\Delta(K)\ge c_3(K)\ge 2\gcan(K),
\]
so the same hypothesis yields
\[
2\br(K)-2\le 2\gcan(K)\le 2c_\Delta(K).
\]
\end{proof}

\begin{remark}
Consequently, a counterexample to either (c2) or (c3) must satisfy
$\br(K)\ge \gcan(K)+2$. This is a useful structural filter because it depends
only on established arrows of Figure~\ref{fig:graph}, not on the numerical
screen.
\end{remark}

\begin{corollary}\label{cor:c2c3genus}
Candidates \textup{(c2)} and \textup{(c3)} hold for every non-trivial knot $K$
satisfying
\[
\br(K)\le g(K)+1.
\]
This class is closed under connected sum and contains every torus knot. Hence
(c2) and (c3) hold for every connected sum of torus knots.
\end{corollary}

\begin{proof}
Since $g(K)\le \gcan(K)$, the hypothesis implies
$\br(K)\le \gcan(K)+1$, so Proposition~\ref{prop:c2c3gc} applies.
For $K=T(p,q)$ with $2\le p<q$ and $\gcd(p,q)=1$, Schubert's bridge-number
formula \cite{Schubert} and the genus computation from the standard positive
braid surface \cite{Cromwell89} give
\[
\br(K)=p,\;\;\; g(K)=\frac{(p-1)(q-1)}2,
\]
and therefore
\[
g(K)+1-\br(K)=\frac{(p-1)(q-3)}2\ge0.
\]
Finally, Schubert's connected-sum formula
$\br(K_1\#K_2)=\br(K_1)+\br(K_2)-1$ \cite{Schubert}, together with
$g(K_1\#K_2)=g(K_1)+g(K_2)$ \cite{Schubert49}, shows that
$\br\le g+1$ is preserved by connected sum. Thus, all connected sums of torus
knots satisfy the hypothesis.
\end{proof}

\subsection{Free surfaces and ribbon comparisons}

\begin{definition}[Homogeneous diagrams]\label{def:homogeneous}
The signed Seifert graph of an oriented diagram has one vertex for each
Seifert circle and one edge, carrying the crossing sign, for each crossing.
A block is a maximal connected subgraph without a cut vertex; in particular,
a bridge edge is a block. The diagram is \emph{homogeneous} if each block has
edges of only one sign. A knot is homogeneous if it
admits such a diagram \cite{Cromwell89}. Alternating diagrams and positive
diagrams are homogeneous.
\end{definition}

\begin{lemma}[Common homogeneous-knot input]\label{lem:homogeneous-data}
If $K$ is homogeneous, then
\[
\gcan(K)=\gfree(K)=g(K),\;\;\;
\spD(K)=\degPz(K)=2g(K).
\]
\end{lemma}
\begin{proof}
Cromwell's homogeneous-link theorem gives a minimal-genus canonical Seifert
surface and $\deg_z\nabla_K=2g(K)$ \cite{Cromwell89}. Thus
$g\le\gfree\le\gcan\le g$, and the Alexander--Conway substitution gives
$\spD=2g$. Specialisation and Morton's degree bound \cite{Morton} now yield
$2g=\deg_z\nabla_K\le\degPz\le2\gcan=2g$.
This lemma assembles the standard inputs used below; it is not claimed as a
new homogeneous-link theorem.
\end{proof}

\begin{proposition}\label{prop:c4}
Candidate \textup{(c4)} holds for every fibred knot.
\end{proposition}

\begin{proof}
Let $K$ be fibred with compact fibre $F$, a Seifert surface of genus $g(K)$,
and let $X=S^3\setminus\operatorname{int}\nu K$ be the knot exterior. Cutting
the mapping torus $X$ open along the fibre gives $F\times I$. Equivalently,
the exterior of a regular neighbourhood of $F$ in $S^3$ is a handlebody
homeomorphic to $F\times I$. Hence its fundamental group is free of rank
$2g(K)$, so $F$ is a free Seifert surface and $\gfree(K)\leq g(K)$. The
reverse inequality $g(K)\leq\gfree(K)$ is immediate from the definitions;
therefore $\gfree(K)=g(K)$. The established skein-depth bound for a knot is
relation $27$ of Figure~\ref{fig:graph}, namely $2g(K)\le \td(K)$
\cite{SchThom}. Therefore, $2\gfree(K)\le \td(K)$.
\end{proof}

\begin{proposition}\label{prop:c4hom}
Candidate \textup{(c4)} holds for every homogeneous knot.
\end{proposition}
\begin{proof}
Lemma~\ref{lem:homogeneous-data}, applied to a homogeneous diagram in the sense
of Definition~\ref{def:homogeneous}, gives $\gfree=g$. Relation $27$ therefore
gives $2\gfree(K)=2g(K)\le\td(K)$.
\end{proof}

The next proposition is a reformulation using the classical comparison
$\gr\leq\urw$, already included as relation~$30$; compare
\cite[Definition~5 and Lemmas~9--10, pp.~319--320]{Shibuya}.
Under the smooth Slice--Ribbon Conjecture the target classes in the definitions
of $\us$ and $\urw$ coincide, so these two crossing-change distances are equal.
The explicit smooth construction below also records the geometric reason for
the required ribbon-genus bound.

\begin{proposition}\label{prop:c5}
The universal validity of candidate \textup{(c5)}, equivalently
\[
\gr(K)\le \us(K)\;\;\; \text{for every knot } K,
\]
is equivalent to the smooth Slice--Ribbon Conjecture.
\end{proposition}

\begin{proof}
Assume first that $\gr(K)\le \us(K)$ for every knot. If $K$ is slice, then
$\us(K)=0$, so $\gr(K)=0$. Thus, $K$ bounds a ribbon disk and is ribbon. Hence
the Slice--Ribbon Conjecture follows.

Conversely, assume the Slice--Ribbon Conjecture; only its validity for the
single knot $J$ produced below is used. Put $n=\us(K)$ and choose $n$ crossing
changes taking $K$ to a slice knot $J$. A crossing change can be realised by two
oriented band attachments \cite[Section~1.4]{JMZ}; equivalently, it gives a
connected genus-one cobordism between the two knots containing only saddle
critical points. Reversing the chosen sequence when necessary and concatenating the
$n$ local cobordisms gives a cobordism from $J$ to $K$ with only saddle points;
it is connected, being a concatenation of connected cobordisms between knots,
and its genus is $n$, since gluing two connected cobordisms between knots along
a single circle adds genera. By Slice--Ribbon, $J$ bounds a ribbon disk. Place
the genus-$n$ saddle-only cobordism in a sufficiently thin collar of
$S^3=\partial B^4$ and glue it to this disk. The ribbon disk has no interior
radial maxima, and the added cobordism contributes only index-one critical
points; after the standard collar reparametrisation, the glued genus-$n$ surface
therefore has no interior radial maxima. It is a ribbon surface bounded by $K$.
Consequently $\gr(K)\le n=\us(K)$.
\end{proof}

\begin{remark}
Thus, (c5) should not be regarded merely as a small-crossing data conjecture: in
the smooth category it packages the full Slice--Ribbon problem.
\end{remark}

\begin{proposition}\label{prop:c6}
Candidate \textup{(c6)} holds for every knot $K$ with
$\cl_4(K)\le \gcan(K)$. In particular it holds whenever $\us(K)\le g(K)$, and
hence for every torus knot.
\end{proposition}

\begin{proof}
Relations $36$ and $3$ of Figure~\ref{fig:graph} give
\[
2c_\Delta(K)\ \ge\ c_3(K)\ \ge\ 2\gcan(K),
\]
so $c_\Delta(K)\ge \gcan(K)$, and the hypothesis yields
$2\cl_4(K)\le 2\gcan(K)\le 2c_\Delta(K)$, which is (c6).

For the second assertion, relation $12$ gives $\cl_4(K)\le \us(K)$, while
relations $5$ and $6$ give $g(K)\le \gfree(K)\le \gcan(K)$; hence
$\us(K)\le g(K)$ implies $\cl_4(K)\le \gcan(K)$. Finally, let $K=T(p,q)$ with $2\le p<q$ and $\gcd(p,q)=1$. Ozawa's torus-knot
evaluation \cite[Theorem~2.9]{Ozawa} gives $a(K)=(p-1)(q-1)/2$, and we use only
its elementary half, the upper bound
\[
a(K)\ \le\ \frac{(p-1)(q-1)}2 ,
\]
which is proved there by an explicit ascending diagram; the reverse inequality
in that theorem passes through the solution of Milnor's conjecture and is not
needed. Together with relation $32$, $u(K)\le a(K)$, and the genus
formula $g(K)=(p-1)(q-1)/2$ coming from the standard positive braid surface
\cite{Cromwell89}, this gives $u(K)\le g(K)$. Since
$\us(K)\le\urw(K)\le u(K)$, we obtain $\us(K)\le g(K)$.
Mirroring changes none of these quantities.
\end{proof}

\begin{remark}
The hypothesis of Proposition~\ref{prop:c6} is not restricted to torus knots;
for example, the sufficient condition $u(K)\le g(K)$ implies it. Failure of
this sufficient condition does not imply failure of the actual hypothesis
$\cl_4(K)\le\gcan(K)$, and still less failure of (c6). Thus, the values
$g(7_4)=1$ and $u(7_4)=2$ merely show that the elementary test $u\le g$ is
inconclusive for $7_4$; they do not decide the sharper criterion. The unresolved
universal statement (c6) is kept separate from all these sufficient tests.
\end{remark}

\subsection{Polynomial and homological criteria}

\begin{proposition}\label{prop:c1alt}
Candidate \textup{(c1)} holds for every non-trivial alternating knot.
\end{proposition}

\begin{proof}
Let $K$ be a non-trivial alternating knot. Ozawa's relations $33$ and $24$
of Figure~\ref{fig:graph} give
\[
2\br(K)-2\le2a(K)\le c(K).
\]
The alternating identity \eqref{eq:altF} of Remark~\ref{rem:kauffman},
obtained from Thistlethwaite's adequate-link estimate
\cite{Thistlethwaite88b,Cromwell98,ParkSeo} and relations $37$ and $45$,
gives $c(K)=\spFa(K)$. Therefore
\[
2\br(K)-2\le c(K)=\spFa(K).\qedhere
\]
\end{proof}

\begin{remark}
The alternating identity \eqref{eq:altF} is not part of the graph: the
displayed arrows $37$ and $45$ only give
$\spFa\le\alpha-2\le c$, in the opposite direction to the additional
lower bound used above. Proposition~\ref{prop:c1alt} therefore does not
make (c1) a transitive consequence of the established relations, and (c1)
retains its conjectural content on non-alternating knots. As explained in
Remark~\ref{rem:kauffman}, \eqref{eq:altF} constrains the imported columns;
it is not an additional propagation rule here.
\end{remark}

For the torus-knot case of (c7), a stronger skein-depth evaluation is already
known. Kaplan, Krcatovich and O'Brien prove that a strictly positive or negative
$n$-braid word of length $\ell$ has closure of depth $\ell-n+1$
\cite[Theorem~1 in the arXiv version]{Kaplan15}.
Applying this to $(\sigma_1\cdots\sigma_{p-1})^q$ gives
\[
 \td(T(p,q))=(p-1)(q-1)=2g(T(p,q))
 \;\;\; (2\leq p<q,\ \gcd(p,q)=1).
\]
Thus, the (c7) assertion below is an application of a known exact depth
calculation and the classical Jones-polynomial formula, rather than a new
torus-knot depth evaluation.

\begin{proposition}\label{prop:torus}
Candidates \textup{(c2)}, \textup{(c3)} and \textup{(c7)} hold for every torus
knot.
\end{proposition}

\begin{proof}
The assertions (c2) and (c3) are already contained in
Corollary~\ref{cor:c2c3genus}. It remains to prove (c7).
The quantities involved are invariant under mirroring, so take $K=T(p,q)$
with $2\le p<q$ and $\gcd(p,q)=1$. The standard positive braid surface and
Cromwell's theorem give $2g(K)=(p-1)(q-1)$ \cite{Cromwell89}.
For (c7), Jones's formula \cite{Jones87}
\[
V_{T(p,q)}(t)\ =\ t^{(p-1)(q-1)/2}\,
\frac{1-t^{p+1}-t^{q+1}+t^{p+q}}{1-t^{2}}
\]
exhibits the Jones polynomial as $t^{(p-1)(q-1)/2}$ times a polynomial in $t$
with constant term $1$ and leading term $-t^{\,p+q-2}$; hence
$\spV(T(p,q))=p+q-2$. Since $p+q\ge 5$, we have
$\lceil (p+q-2)/2\rceil \le p+q-3$, and
\[
(p-1)(q-1)-(p+q-3)\ =\ (p-2)(q-2)\ \ge\ 0
\]
gives $p+q-3\le 2g(K)$. Therefore, using relation $27$ once more,
\[
\Big\lceil \tfrac{\spV(K)}{2}\Big\rceil\ \le\ p+q-3\ \le\ 2g(K)\ \le\ \td(K).
\]
For the trefoil $T(2,3)$ the final chain is an equality. The earlier
(c2) and (c3) assertions also hold with equality there, consistently with the
certified equality witness for (c3) recorded above.
\end{proof}

\begin{proposition}\label{prop:positive}
For every positive knot $K$ (that is, a knot admitting a diagram all of whose
crossings are positive), candidates \textup{(c8)} and \textup{(c9)} hold with
equality:
\[
2|\tau(K)|\ =\ |s(K)|\ =\ \degPz(K)\ =\ 2g(K).
\]
\end{proposition}

\begin{proof}
Rasmussen proved $s(K)=2g_4(K)=2g(K)$ for positive knots
\cite{Rasmussen}. Abe's characterisation of positive knots gives, in addition,
$\tau(K)=s(K)/2=g_4(K)=g(K)$
\cite[Theorem~1.3 in the arXiv version]{Abe11}.
A positive diagram is homogeneous, so
Lemma~\ref{lem:homogeneous-data}, using Cromwell's theorem and Morton's bound,
gives $\degPz(K)=2g(K)$. These established equalities prove the assertion.
\end{proof}

\begin{remark}
Proposition~\ref{prop:positive} supplies theorem-level equality witnesses for
(c8) and (c9): every positive row of \textsf{NewDB5} on which the columns
$2|\tau|$, $|s|$ and $\degPz$ are exact is a certified equality witness for
both candidates simultaneously.
\end{remark}

\begin{proposition}\label{prop:c8c9hom}
Candidates \textup{(c8)} and \textup{(c9)} hold for every homogeneous knot.
More generally, they hold for every knot satisfying $\degPz(K)=2g(K)$.
For a homogeneous knot, equality in (c8) occurs exactly when $g_4(K)=g(K)$
and the Floer slice-genus bound is sharp; similarly, equality in (c9) occurs
exactly when $g_4(K)=g(K)$ and the Rasmussen slice-genus bound is sharp.
\end{proposition}

\begin{proof}
For every knot, relations $42$ and $41$, together with $g_4(K)\le g(K)$, give
\[
2|\tau(K)|\le 2g_4(K)\le 2g(K),\;\;\;
|s(K)|\le 2g_4(K)\le 2g(K).
\]
Thus, the more general assertion follows immediately whenever
$\degPz(K)=2g(K)$.

For homogeneous knots, Lemma~\ref{lem:homogeneous-data} supplies
$\degPz(K)=2g(K)$.
The equality criterion follows because equality at the two ends of either chain
above forces, and is forced by, equality in both intermediate inequalities.
\end{proof}

\begin{remark}[Comparison with Abe's homogeneous-knot results]
Abe proves $2\tau(K)=s(K)$ for homogeneous knots
\cite[Theorem~1.2 and the following discussion in the arXiv version]{Abe11}.
Consequently (c8) and (c9) are equivalent on this class. For a non-trivial
homogeneous knot, equality in either comparison occurs precisely when $K$ or
its mirror is positive. Indeed, equality and
Lemma~\ref{lem:homogeneous-data} give $|\tau(K)|=g(K)$. After mirroring when
necessary, the slice-genus bounds give
$\tau=s/2=g_4=g$, and Abe's positive-knot characterisation applies
\cite[Corollary~1.4 in the arXiv version]{Abe11}.
The converse follows from Proposition~\ref{prop:positive} and mirroring.
\end{remark}

\begin{definition}\label{def:sigmathin}
Call $K$ \emph{simultaneously signature-thin} if its unreduced rational
Khovanov homology is supported on the two diagonals
\[
 Kh^{i,j}(K;\mathbb Q)\ne0\;\;\Longrightarrow\;\;
 j-2i\in\{-\sigma(K)-1,-\sigma(K)+1\},
\]
and its hat knot Floer homology over $\mathbb F_2$ is supported on
\[
 \widehat{HFK}_{M}(K,A;\mathbb F_2)\ne0\;\;\Longrightarrow\;\;
 M-A=\frac{\sigma(K)}2.
\]
Here $i,j$ are the homological and quantum gradings, and $M,A$ are the
Maslov and Alexander gradings. We use $s$ over characteristic zero.
The term fixes both coefficient fields and grading normalisation; support
on an unspecified number of diagonals is not sufficient.
\end{definition}

\begin{lemma}[Signature-thin concordance invariants]\label{lem:thin-values}
A simultaneously signature-thin knot satisfies
$s(K)=-\sigma(K)$ and $\tau(K)=-\sigma(K)/2$.
\end{lemma}
\begin{proof}
Lee's filtered deformation and Rasmussen's normalisation give two surviving
knot classes in homological degree zero, at quantum filtration levels
$s-1$ and $s+1$ \cite{Lee05,Rasmussen}. On the stated two Khovanov diagonals,
the associated-graded survivors in that degree can only have levels
$-\sigma-1$ and $-\sigma+1$, forcing $s=-\sigma$.
For knot Floer homology, the spectral sequence from $\widehat{HFK}(K)$ to
$\widehat{HF}(S^3)$ leaves a generator of Maslov grading zero, whose Alexander
filtration level is $\tau(K)$ \cite{OzsSzaGenus,OzsSzaAlt}. The support equation
$M-A=\sigma/2$ therefore gives $\tau=-\sigma/2$.
\end{proof}

\begin{proposition}\label{prop:c8c9thin}
Candidates \textup{(c8)} and \textup{(c9)} hold for every simultaneously
signature-thin knot.
\end{proposition}

\begin{proof}
By Lemma~\ref{lem:thin-values},
\[
|s(K)|=2|\tau(K)|=|\sigma(K)|.
\]
Relations $34$ and $40$ then give
\[
|s(K)|=2|\tau(K)|=|\sigma(K)|\le \spD(K)\le \degPz(K),
\]
which proves (c8) and (c9).
\end{proof}

\begin{corollary}\label{cor:qa}\label{prop:c8c9}
Candidates \textup{(c8)} and \textup{(c9)} hold for every quasi-alternating
knot, and in particular for every alternating knot.
\end{corollary}

\begin{proof}
Manolescu and Ozsv\'ath prove that quasi-alternating links are Khovanov
homologically $\sigma$-thin over $\mathbb Z$ in the reduced theory and Floer
homologically $\sigma$-thin over $\mathbb F_2$
\cite[Theorems~1 and~2 in the arXiv version]{ManOzs}. For a knot, their
Khovanov grading convention is half the integral quantum grading used in
Definition~\ref{def:sigmathin}; hence reduced $\sigma$-thinness means that
the reduced theory is supported on $j-2i=-\sigma(K)$ in our integral grading.
Over $\mathbb Q$, the reduced--unreduced long exact sequence of
Asaeda--Przytycki \cite{AsaedaPrzytycki}, with its two quantum shifts $\pm1$,
therefore confines the unreduced rational Khovanov homology to
\[
 j-2i\in\{-\sigma(K)-1,-\sigma(K)+1\}.
\]
For knot Floer homology, Manolescu and Ozsv\'ath use
$\delta=A-M=-\sigma(K)/2$, equivalently
$M-A=\sigma(K)/2$. Thus, every quasi-alternating knot is simultaneously
signature-thin in the sense of Definition~\ref{def:sigmathin}, and
Proposition~\ref{prop:c8c9thin} applies. Alternating knots are
quasi-alternating.
\end{proof}

\begin{remark}
Propositions~\ref{prop:c8c9hom} and \ref{prop:c8c9thin} provide two
independent sufficient mechanisms for (c8) and (c9). The first uses the
three- and four-dimensional chain
$2|\tau|,|s|\le 2g_4\le 2g=\degPz$; the second uses
$2|\tau|=|s|=|\sigma|\le \spD\le \degPz$.
Thus, a counterexample to either candidate must be both non-homogeneous and not
simultaneously signature-thin.
\end{remark}

\subsection{Restrictions on counterexamples}

\begin{proposition}\label{prop:filters}
The established graph gives the following necessary conditions for a
counterexample to the remaining candidates.
\begin{enumerate}[label=\textup{(\roman*)},leftmargin=2.4em]
\item A counterexample to (c2) or to (c3) satisfies
$\br(K)\ge \gcan(K)+2$.
\item A counterexample to (c4) satisfies $\gfree(K)>g(K)$; in particular it is
neither fibred nor homogeneous.
\item A counterexample to (c6) satisfies $\cl_4(K)>\gcan(K)\ge g(K)$, and
consequently $\us(K)>\gcan(K)$, $\uc(K)>\gcan(K)$ and $u(K)>\gcan(K)$.
\item A counterexample to (c7) satisfies $\spV(K)\ge 4g(K)+1$.
\item A counterexample to (c8) satisfies $2|\tau(K)|>|\sigma(K)|$, and a
counterexample to (c9) satisfies $|s(K)|>|\sigma(K)|$.
\end{enumerate}
\end{proposition}

\begin{proof}
Part (i) is Proposition~\ref{prop:c2c3gc}. If (c4) fails, then relation $27$
gives
\[
2\gfree(K)>\td(K)\ge 2g(K),
\]
so $\gfree(K)>g(K)$; Propositions~\ref{prop:c4} and \ref{prop:c4hom} then
exclude fibred and homogeneous knots. If (c6) fails, then relations $36$ and
$3$ give
\[
2\cl_4(K)>2c_\Delta(K)\ge c_3(K)\ge 2\gcan(K)\ge 2g(K),
\]
so $\cl_4(K)>\gcan(K)$; since relations $12$, $25$ and $35$ give
$\cl_4\le \us\le \urw\le u$ and relations $46$ and $31$ give
$\cl_4\le \uc\le u$, the three remaining inequalities of (iii) follow. If (c7)
fails, then
\[
\Big\lceil\frac{\spV(K)}2\Big\rceil>\td(K)\ge 2g(K).
\]
The left-hand side is an integer, hence it is at least $2g(K)+1$, which is
equivalent to $\spV(K)\ge 4g(K)+1$. Finally, relations $40$ and $34$ give
$\degPz(K)\ge \spD(K)\ge |\sigma(K)|$. Therefore, failure of (c8) forces the stated strict inequality involving
$\tau$, while failure of (c9) forces the corresponding inequality involving
$s$.
\end{proof}

\begin{remark}\label{rem:reduce}
Propositions~\ref{prop:c2c3}, \ref{prop:c2c3gc}, \ref{prop:c4},
\ref{prop:c4hom}, \ref{prop:c6}, \ref{prop:c1alt},
\ref{prop:torus}, \ref{prop:positive}, \ref{prop:c8c9hom},
\ref{prop:c8c9thin} and \ref{prop:filters}, together with
Corollaries~\ref{cor:c2c3genus} and \ref{cor:qa} reduce the first
counterexample searches for the retained candidates as follows. For (c1), it
suffices to search among non-alternating knots. For (c2) and (c3), it suffices
to consider knots with bridge number at least four and, more sharply,
$\br\ge \gcan+2$, thereby excluding in particular all connected sums of torus
knots. For (c4), the knot must be both non-fibred and non-homogeneous. For
(c6), it must be a non-torus knot with $\cl_4>\gcan$, hence with
$\us>\gcan$ and $u>\gcan$. For (c7), it must be a non-torus knot with
$\spV\ge 4g+1$. Finally, a counterexample to (c8) or (c9) must be
non-homogeneous and not simultaneously signature-thin; for (c8) it must also
satisfy $2|\tau|>|\sigma|$, and for (c9) it must satisfy
$|s|>|\sigma|$. These are theorem-level filters and do not use the
computational witness counts.
\end{remark}

\appendix
\section{Short definitions of the invariants}\label{app:defs}

\subsection*{Crossing-type and diagrammatic invariants}
\ \\
\noindent\textbf{1. Crossing number $c(K)$.} The crossing number $c(K)$ of a
knot $K$ is the minimum number of ordinary double crossings among all regular
planar diagrams representing the knot. A regular diagram is obtained from a
generic projection of the knot into the plane so that all singularities are
transverse double points equipped with over-under crossing information.

\noindent\textbf{2. Triple-crossing number $c_3(K)$.} A triple-crossing
projection of $K$ is a projection all of whose singularities are triple points,
with no ordinary double crossings. A triple crossing is a point where exactly
three strands intersect transversely, together with a specified height ordering
of the top, middle, and bottom strands. Every knot admits such a projection
\cite{Adams}. The triple-crossing number $c_3(K)$ is the minimum number of
triple crossings among all such projections.

\noindent\textbf{3. Delta-crossing number $c_\Delta(K)$.}
A delta-diagram is an ordinary diagram whose crossings are partitioned into
pairwise disjoint local three-crossing blocks, each equal to the delta-crossing
tangle of Nakanishi--Sakamoto--Satoh, or its reflected version
\cite{NSS15,Jab23}. There are no ordinary crossings outside those blocks.
The number $c_\Delta(K)$ is the minimum number of blocks over such diagrams
representing $K$. The word ``delta'' refers to this specified local tangle,
not to an arbitrary grouping of three ordinary crossings or to a triple point.
Resolving each block as three ordinary crossings gives
$c(K)\le3c_\Delta(K)$.

\noindent\textbf{4. Skein tree depth $\td(K)$.} An ordinary oriented binary
skein tree has a diagram of $K$ at its root and unlink diagrams at its leaves.
At each internal vertex, a crossing is chosen; one child is obtained by
switching that crossing and the other by applying its oriented smoothing. These are the two terms obtained by
solving the HOMFLY--PT skein relation of Convention~\ref{conv:poly} for the
parent. Reidemeister moves and diagram isotopies may be made without cost
between resolutions. The depth is the maximum number of edges in a root-to-leaf
path, counting \emph{both} switching and smoothing edges.
The skein tree depth $\td(K)$ is the minimum of this depth over all such trees
and all diagrams of $K$ \cite{Cromwell89,SchThom,Thompson89,Kaplan15,Jab26}; the same definition
applies to links. No division by a computed polynomial, pruning by cancellation,
or replacement of this two-branch recursion by a different skein relation is
allowed in this definition. The unlink with $r\ge1$ components has polynomial
$((a-a^{-1})/z)^{r-1}$, as used in Theorem~\ref{prop:td-leading}.

\noindent\textbf{5. Trivializing number $\tr(K)$.} Following Hanaki
\cite{Hanaki10,Hanaki14}, start from a knot projection $P$, whose double points
are initially precrossings. A pseudo-diagram $Q$ obtained from $P$ is called
trivial if over/under information has been fixed at some of the precrossings and
every resolution of all remaining precrossings is the unknot. The trivializing
number $\tr(P)$ is the minimum number of crossings whose over/under information
must be fixed in such a trivial pseudo-diagram. If $D$ is a diagram with
underlying projection $P$, set $\tr(D)=\tr(P)$, and define
$\tr(K)=\min\{\tr(D): D \text{ represents } K\}$. Thus, $\tr(K)$ counts the
crossings whose information is prescribed, not the crossings whose information
is forgotten. Hanaki also proved that $\tr(K)$ is even, that $\gcan(K)\le \tr(K)/2$
(relation $28$), and that $\tr(K)\le c(K)-1$ for every non-trivial knot
\cite{Hanaki10,Hanaki14}; the weaker bound $g(K)\le \tr(K)/2$ is \cite{Henrich},
and $u(K)\le \tr(K)/2$ is relation $10$. See Subsection~\ref{sec:next}.

\noindent\textbf{6. Free genus $\gfree(K)$.} The free genus $\gfree(K)$ of a
knot $K$ is the minimum genus among all Seifert surfaces whose complement in
$S^3$ has free fundamental group. Such spanning surfaces are called free
Seifert surfaces; for a standard reference on this invariant see
\cite{Moriah}.

\noindent\textbf{7. Canonical genus $\gcan(K)$.} The canonical genus $\gcan(K)$
is the minimum genus among all Seifert surfaces obtained from Seifert's
algorithm applied directly to knot diagrams of $K$. Every canonical Seifert
surface is free \cite{KobKob}, whence $g\le \gfree \le \gcan$.

\noindent\textbf{8. Genus $g(K)$.} The genus $g(K)$ of a knot $K$ is the minimum
genus among all connected, compact, orientable Seifert surfaces whose boundary
is the knot. A Seifert surface is an orientable surface smoothly embedded in
$S^3$ with boundary equal to $K$.

\noindent\textbf{9. Unknotting number $u(K)$.} The unknotting number $u(K)$ is
the minimum number of crossing changes required to transform the knot into the
unknot. A crossing change replaces a positive crossing with a negative crossing
or vice versa.

\noindent\textbf{10. Band-unknotting number $\ub(K)$.} The band-unknotting
number $\ub(K)$ is the minimum number of oriented band surgeries required to
transform the knot into the unknot. A band surgery removes two small arcs from
the knot and reconnects them using the opposite sides of an embedded band. Since
an oriented band surgery changes the number of components by one, $\ub(K)$ is
even.

\noindent\textbf{11. Band-unlinking number $\ulb(K)$.} The band-unlinking number
$\ulb(K)$ is the minimum number of oriented band surgeries required to transform
the knot into a completely split unlink. With this convention, $\ulb(K)$ is
McDonald's band number $b(K)$ \cite{McDonald}.

\noindent\textbf{12. Clasp number $\cl(K)$.}
Following Shibuya \cite[Definitions~2--3]{Shibuya}, the clasp number is the
minimum, over all \emph{clasp disks}, of the number of clasp singularities. A
clasp disk is a singular spanning disk in $S^3$ whose boundary maps
homeomorphically onto $K$, whose only singularities are
finitely many ordinary clasp double arcs, and which have no ribbon double
arcs, closed double curves, or triple points. For a clasp double arc, its two
preimage arcs in the disk each have one endpoint on the boundary and one in
the interior, in the standard transverse clasp model. Requiring all
singularities to be clasps is essential. Allowing arbitrary singular disks
and counting only their clasp arcs would give a different quantity, for
which relations $4$ and $29$ need not hold.

\noindent\textbf{13. Bridge number $\br(K)$.} The bridge number $\br(K)$ is the
minimum number of local maxima of the knot with respect to a Morse height
function on $S^3$, minimised over all embeddings isotopic to the knot.
Equivalently, it is the minimum number of bridges appearing in a bridge
decomposition of the knot.

\noindent\textbf{14. Braid index $b(K)$.} The braid index $b(K)$ is the minimum
number of strands required to represent the knot as the closure of an element of
the Artin braid group. Explicitly,
$b(K)=\min\{n: K \text{ is the closure of an } n\text{-braid}\}$.

\noindent\textbf{15. Arc index $\alpha(K)$.} The arc index $\alpha(K)$ is the
minimum number of arcs needed in an arc presentation of the knot. Equivalently,
it is the smallest number of half-planes in an open-book decomposition of $S^3$
whose pages contain embedded arcs forming the knot.

\noindent\textbf{16. Mosaic number $m(K)$.} The mosaic number $m(K)$ is the
minimum integer $n$ such that the knot admits a representation as an
$n\times n$ knot mosaic. A knot mosaic is a finite grid of prescribed tiles
assembled according to compatibility rules that encode the knot diagram
combinatorially. We use the standard tile set and matching convention of
\cite{LHLO}.

\noindent\textbf{17. Ascending number $a(K)$.} Following Ozawa
\cite{Ozawa}, the ascending number $a(K)$ is the minimum, over all based
oriented diagrams of $K$, of the number of crossing changes needed to
transform the diagram into a descending diagram. A descending diagram is one
in which, after choosing a base point and orientation, every crossing is first
encountered as an overcrossing. Such diagrams always represent the unknot.

\subsection*{Four-dimensional and concordance invariants}
\ \\
\noindent\textbf{18. Four-dimensional clasp number $\cl_4(K)$.}
This is the minimum number of transverse double points of a smooth proper
self-transverse immersion $f:D^2\looparrowright B^4$ such that
$f^{-1}(S^3)=\partial D^2$, the restriction $f|_{\partial D^2}$ is an embedding
onto $K$, and there is an embedded standard boundary collar. Its only multiple
points are isolated transverse interior double points. The \emph{map} is smooth also at their preimages;
only its image fails to be embedded. Resolving each double point by an
oriented tube gives $g_4(K)\le\cl_4(K)$ \cite{MurYas,OwensStrle}.

\noindent\textbf{19. Slice genus $g_4(K)$.} The slice genus $g_4(K)$, also
called the smooth four-ball genus, is the minimum genus of a smooth, compact,
connected, orientable surface properly embedded in $B^4$ whose boundary is the
knot $K\subset S^3=\partial B^4$. The invariant measures how efficiently the
knot bounds a smooth surface in four dimensions rather than in the three-sphere.
The locally flat analogue $\gtop(K)$ of Convention~\ref{conv:cat} is used only as
an auxiliary quantity in the justification of relation $34$; it is not a vertex
of Figure~\ref{fig:graph}.

\noindent\textbf{20. Knot signature $\sigma(K)$.} The knot signature $\sigma(K)$
is the signature of the symmetrised Seifert form associated with a Seifert
surface for the knot. If $V$ is a Seifert matrix, then
$\sigma(K)=\sgn(V+V^{T})$, where the signature denotes the number of positive
eigenvalues minus the number of negative eigenvalues; the Seifert-form
orientation convention is the one fixed in Convention~\ref{conv:cat}. For knots, the signature is even; it is an integer-valued concordance invariant independent of the chosen
Seifert surface.

\noindent\textbf{21. Doubly slice genus $\gds(K)$.}
This is the minimum genus of a smooth, connected, closed, orientable,
\emph{unknotted} surface $F\subset S^4$ whose transverse intersection with an
equatorial $S^3$ is $K$. Here unknotted means that $F$ bounds a smoothly
embedded three-dimensional handlebody in $S^4$ \cite{LivMeier,McDonald}.
Since $F$ is connected and its intersection with the equator is one circle,
cutting along the equator produces two connected properly embedded surfaces
with boundary $K$ (with opposite boundary orientations), and their genera add
to $g(F)$. Each has genus at least $g_4(K)$; hence
$\gds(K)\geq2g_4(K)$.

\noindent\textbf{22. Ribbon genus $\gr(K)$.}
This is the minimum genus of a smooth, compact, connected, orientable surface
properly embedded in $B^4$ with boundary $K$, for which the radial function can
be chosen Morse with only index-zero and index-one interior critical points.
Such a surface is ribbon; its handle presentation consists of disks and bands,
with no maxima \cite{Shibuya,JMZ}. In particular, $\gr(K)=0$ exactly when $K$
bounds a ribbon disk.

\noindent\textbf{23. Weak ribbon unknotting number $\urw(K)$.}
For the knots considered here, $\urw(K)$ is the minimum number of crossing
changes taking $K$ to a knot bounding a smooth ribbon disk in $B^4$.
This is the knot case of Shibuya's weak ribbon unknotting number
\cite[Definition~5 and Lemma~5]{Shibuya}: his genus-zero ribbon singular
spanning surface, when it has one boundary component, is a disk, and the usual
push-in construction gives an embedded ribbon disk in $B^4$. No assertion
about the equivalence of ribbon and weak ribbon \emph{links} is needed here.
Thus, $\urw(K)=0$ precisely for ribbon knots; $\urw$ is a crossing-change
distance and is not the number of ribbon singularities of a disk.
The two-oriented-band realisation of a crossing change
\cite[Section~1.4]{JMZ}, capped by a ribbon disk at the other end, also
explains directly why $\gr(K)\le\urw(K)$.

\noindent\textbf{24. Slicing number $\us(K)$.} The slicing number $\us(K)$ is
the minimum number of crossing changes required to transform the knot into a
slice knot \cite{OwensStrle}. A slice knot is a knot bounding a smoothly
embedded disk in the four-ball. Thus, the slicing number measures the distance
from the knot to the slice concordance class rather than to the unknot itself.

\noindent\textbf{25. Concordance unknotting number $\uc(K)$.} The concordance
unknotting number is
\[
 \uc(K)=\min\{u(J):J\sim K\},
\]
where $\sim$ denotes smooth knot concordance \cite{OwensStrle}. The invariant
measures the smallest unknotting number achievable within the smooth
concordance class of the knot rather than for the knot itself.

\subsection*{Quantities from knot homologies}
\ \\
\noindent\textbf{26. Torsion order $\Ord_v(K)$.} Following Juh\'asz, Miller
and Zemke \cite[Definition~1.1]{JMZ}, let $K\subset S^3$ be a knot, and let
$HFK^-(K)$ denote the minus version of knot Floer homology, viewed as a
finitely generated module over $\mathbb{F}_2[v]$. The $v$-torsion submodule of
$HFK^-(K)$ is
\[
\Tors_v\!\left(HFK^-(K)\right)
=\{x\in HFK^-(K): v^m x = 0 \text{ for some } m\ge 1\}.
\]
The torsion order of $K$ is defined by
\[
\Ord_v(K)=\min\left\{n\ge 0: v^{n}\cdot \Tors_v\!\left(HFK^-(K)\right)=0\right\}.
\]
Equivalently, $\Ord_v(K)$ is the smallest integer $n$ such that multiplication by
$v^{n}$ annihilates all $v$-torsion elements in $HFK^-(K)$.

\noindent\textbf{27. Ozsv\'ath--Szab\'o invariant $\tau(K)$.}
Let $\mathcal F_j\widehat{CF}(S^3)$ be the Alexander-filtered subcomplex
determined by $K$. Then $\tau(K)$ is the least integer $j$ for which its
inclusion induces a nonzero map on homology to
$\widehat{HF}(S^3)\cong\mathbb F_2$ \cite{OzsSzaGenus}. This is an integer-valued smooth concordance invariant, with
$\tau(T(2,3))=1$ in our convention. The tabulated values imported for the
$2|\tau|$ column originate in Livingston's computations and their
successors \cite{Livingston04,KnotInfo}.

\noindent\textbf{28. Rasmussen invariant $s(K)$.}
Over $\mathbb Q$, the two-dimensional Lee homology of a knot carries quantum
filtration levels $s_{\min}(K)$ and $s_{\max}(K)$, with
$s_{\max}-s_{\min}=2$. Rasmussen defines
$s(K)=(s_{\min}+s_{\max})/2$ \cite{Lee05,Rasmussen}. It is an even integer and a
smooth concordance invariant; our normalisation has $s(T(2,3))=2$.

\subsection*{Quantities from knot polynomials}
\ \\
\noindent\textbf{29. Span of the Alexander polynomial $\spD(K)$.} The span of
the Alexander polynomial is defined by
$\spD(K)=\deg_{\max}\Delta_K(t)-\deg_{\min}\Delta_K(t)$, where $\Delta_K(t)$ is
the Alexander polynomial of the knot. The invariant is the total width of the
exponents appearing in the Alexander polynomial and gives a numerical measure
of algebraic complexity. Since the Alexander polynomial is symmetric up to
multiplication by powers of $t$, the span is independent of normalisation
choices.

\noindent\textbf{30. Span of the Jones polynomial $\spV(K)$.} The span of the
Jones polynomial is defined by $\spV(K)=\deg_{\max}V_K(t)-\deg_{\min}V_K(t)$,
where $V_K(t)$ denotes the Jones polynomial of the knot. By the
Kauffman--Murasugi--Thistlethwaite theorem
\cite{Kauffman87,Murasugi87,Thistlethwaite}, one has $\spV(K)\le c(K)$.
For knots, equality holds exactly for alternating knots; the converse follows
from Murasugi's characterisation \cite[Theorem~4]{Murasugi87}. See
Remark~\ref{rem:kmt} for the distinction between this knot-type statement and
the stronger diagram-level assertion, which requires a primeness hypothesis.

\noindent\textbf{31. HOMFLY-PT $z$-degree $\degPz(K)$.} The invariant
$\degPz(K)$ denotes the maximal degree in the variable $z$ appearing in the
HOMFLY-PT polynomial $P_K(a,z)$, in the normalisation of
Convention~\ref{conv:poly}. The HOMFLY-PT polynomial is characterised by its
skein relation and simultaneously generalises both the Alexander and Jones
polynomials.

\noindent\textbf{32. HOMFLY-PT $a$-spread $\spPa(K)$.} The HOMFLY-PT $a$-spread
is defined by
$\spPa(K)=\deg^{a}_{\max}P_K(a,z)-\deg^{a}_{\min}P_K(a,z)$.

\noindent\textbf{33. Kauffman polynomial $a$-spread $\spFa(K)$.}
We use the ambient-isotopy, two-variable Kauffman polynomial, normalised by
$F_U(a,z)=1$, in the convention of \cite{MortonBeltrami,KnotInfoPolynomials}.
It is the regular-isotopy Kauffman polynomial of an oriented diagram multiplied
by its writhe-correcting monomial in $a$. Define
$\spFa(K)=\deg^a_{\max}F_K-\deg^a_{\min}F_K$.
The writhe correction, mirroring, and the usual original/Dubrovnik convention
change do not affect this spread. This quantity is not the Jones polynomial
or the HOMFLY--PT $a$-spread.

\section{Machine-readable form of Figure~\ref{fig:graph}}\label{app:table}

A picture is not a machine-readable input for the propagation, so we record the
arrow set of Figure~\ref{fig:graph} here as well, with columns
\texttt{(id, source, target, reference)} in the source/target convention
$\texttt{source}\ \ge\ \texttt{target}$ of Section~\ref{sec:network}. Viewed as a directed
graph, the arrow set in Table~\ref{tab:rel} is acyclic, and no arrow is implied
by the remaining $46$:
every displayed relation carries information that is not recoverable from the
others by transitivity alone. Table~\ref{tab:rel} and Figure~\ref{fig:graph}
encode the same relations, including the orientation of every arrow.

For the nine candidates of Section~\ref{sec:candidates}, orient each
proposed edge from its right-hand side to its left-hand side. Adjoining
these nine edges leaves the graph acyclic, and deleting any one of them
leaves no directed path from its source to its target. This is joint
transitivity independence relative to the selected $47$ arrows, not
validity as knot inequalities. The $47$ established arrows are irredundant within their own graph.
They need not remain so after adding proposed edges, and in fact exactly two
of them do not: adjoining all nine candidate edges makes arrow $27$
($\td\ge 2g$) redundant, through
$\td\ge 2\gfree\ge 2g$ using (c4) and relation $6$, and makes arrow $30$
($2\urw\ge 2\gr$) redundant, through $2\urw\ge 2\us\ge 2\gr$ using relation
$25$ and (c5). This is a statement about the directed graph only. Neither
composite is available as a proof, since (c4) and (c5) are unproved in
general; in particular the appeals to relation $27$ in
Propositions~\ref{prop:c4}, \ref{prop:c4hom}, \ref{prop:torus} and
\ref{prop:filters}, and the appeal to relation $30$ before
Proposition~\ref{prop:c5}, are not circular. No candidate edge is used in the
propagation that constructs \textsf{NewDB5}.

\begin{table}[htbp]
\centering\footnotesize
\begin{tabular}{clll}
\toprule
id & source & target & ref. \\
\midrule
$1$ & $c$ & $\td$ & \cite{Cromwell89} \\
$2$ & $c$ & $c_3$ & \cite{Adams} \\
$3$ & $c_3$ & $2\gcan$ & \cite{Jab20} \\
$4$ & $2\cl$ & $2g$ & \cite{Shibuya} \\
$5$ & $2\gcan$ & $2\gfree$ & \cite{KobKob} \\
$6$ & $2\gfree$ & $2g$ & \cite{Moriah} \\
$7$ & $2\gcan$ & $\degPz$ & \cite{Morton} \\
$8$ & $2g$ & $\spD$ & \cite{Giller} \\
$9$ & $c_3$ & $\lceil\spV/2\rceil$ & \cite{Adams} \\
$10$ & $\tr$ & $2u$ & \cite{Henrich} \\
$11$ & $2\cl_4$ & $2g_4$ & \cite{MurYas} \\
$12$ & $2\us$ & $2\cl_4$ & Section~\ref{sec:network}, item~12 \\
$13$ & $2g_4$ & $|\sigma|$ & \cite{KaufTaylor} \\
$14$ & $\alpha-2$ & $2b-2$ & \cite{Cromwell95} \\
$15$ & $2u$ & $\ub$ & \cite{HNT} \\
$16$ & $2g$ & $\ub$ & \cite{JMZ} \\
$17$ & $\ub$ & $\ulb$ & \cite{JMZ} \\
$18$ & $\ulb$ & $2\gr$ & \cite{JMZ} \\
$19$ & $2\gr$ & $2g_4$ & Section~\ref{sec:network}, item~19 \\
$20$ & $\ulb$ & $\gds$ & \cite{McDonald} \\
$21$ & $\gds$ & $2g_4$ & \cite{LivMeier} \\
$22$ & $2b-2$ & $2\br-2$ & \cite{Schubert} \\
$23$ & $2b-2$ & $\spPa$ & \cite{FranksWilliams,Morton} \\
$24$ & $c$ & $2a$ & \cite{Ozawa} \\
\bottomrule
\\
\end{tabular}
\;\;\;
\begin{tabular}{clll}
\toprule
id & source & target & ref. \\
\midrule
$25$ & $2\urw$ & $2\us$ & \cite{Shibuya} \\
$26$ & $c$ & $\tr$ & \cite{Hanaki14} \\
$27$ & $\td$ & $2g$ & \cite{SchThom} \\
$28$ & $\tr$ & $2\gcan$ & \cite{Hanaki14} \\
$29$ & $2\cl$ & $2u$ & \cite{Shibuya} \\
$30$ & $2\urw$ & $2\gr$ & \cite{Shibuya} \\
$31$ & $2u$ & $2\uc$ & \cite{OwensStrle} \\
$32$ & $2a$ & $2u$ & \cite{Ozawa} \\
$33$ & $2a$ & $2\br-2$ & \cite{Ozawa} \\
$34$ & $\spD$ & $|\sigma|$ & \cite{Feller,KaufTaylor} \\
$35$ & $2u$ & $2\urw$ & \cite{Shibuya} \\
$36$ & $2c_\Delta$ & $c_3$ & \cite{Jab23} \\
$37$ & $\alpha-2$ & $\spFa$ & \cite{MortonBeltrami} \\
$38$ & $\alpha-2$ & $m-1$ & \cite{LHLO} \\
$39$ & $\td$ & $\degPz$ & \cite{Jab26} \\
$40$ & $\degPz$ & $\spD$ & Section~\ref{sec:network}, item~40 \\
$41$ & $2g_4$ & $|s|$ & \cite{Rasmussen} \\
$42$ & $2g_4$ & $2|\tau|$ & \cite{OzsSzaGenus} \\
$43$ & $2c_\Delta$ & $\td$ & \cite{Jab26} \\
$44$ & $\ulb$ & $\Ord_v$ & \cite{JMZ} \\
$45$ & $c$ & $\alpha-2$ & \cite{BaePark} \\
$46$ & $2\uc$ & $2\cl_4$ & \cite{OwensStrle} \\
$47$ & $2\br-2$ & $\Ord_v$ & \cite{JMZ} \\
\bottomrule
\\
\end{tabular}
\caption{The $47$ arrows of Figure~\ref{fig:graph}, in the convention
$\text{source}\ge\text{target}$.}
\label{tab:rel}
\end{table}

\subsection*{Data Availability}

The database \textsf{NewDB5} is available in the author's database archive 
\url{https://drive.google.com/drive/folders/1mJk2DDfyKQC7XpIRDOYj4YELWWJ9FN1C}.

\section*{Acknowledgements}

The author used OpenAI's ChatGPT for editorial assistance, literature searches, exploratory mathematical checking, and drafting computational verification code.

\end{document}